\documentclass[reqno]{amsart}
\usepackage{amsmath,amsthm,amssymb}
\usepackage{latexsym}
\usepackage{eucal}

\newtheorem{theorem}{Theorem}[section]
\newtheorem{lemma}{Lemma}[section]
\newtheorem{corollary}{Corollary}[section]

\newtheorem{proposition}{Proposition}[section]

\theoremstyle{definition}

\newtheorem{open problem}{Open Problem}

\def\cal{\mathcal}
\let\Re=\undefined
\DeclareMathOperator{\Re}{Re}
\let\Im=\undefined
\DeclareMathOperator{\Im}{Im}

\begin{document}
\title[The generic behavior of solutions \ldots
 ]{The generic behavior of solutions to some evolution equations: asymptotics and Sobolev norms}
\author{Sergey A. Denisov}
\address{
\begin{flushleft}
University of Wisconsin--Madison\\  Mathematics Department\\
480 Lincoln Dr., Madison, WI, 53706, USA\\
  denissov@math.wisc.edu
\end{flushleft}
  }
 \maketitle
\begin{abstract}
In this paper, we study the generic behavior of the solutions to a
large class of evolution equations. The Schr\"odinger evolution is
considered as an application.
\end{abstract} \vspace{1cm}

In this paper, we develop methods to control the generic behavior of
solutions to various evolution equations. The word ``generic" will
refer to the ``coupling constant" that appears in front of the
diagonal operator (a differential operator in most applications)
which is perturbed by a potential.  We were motivated by recent
results where the behavior of Sobolev norms was studied for various
evolution equations \cite{b1,pert, w1}. The general situation of
fixed coupling constant was studied in many papers (e.g., \cite{Nen}
and references therein).

The structure of the paper is as follows. In the first section, we
prove simple results for the ``typical" behavior of solutions to
 $2\times 2$ system of differential equations which preserve $\ell^2$
 norm of the initial value in Cauchy problem. The section 2 deals with ``integrable"
 case when the
 transport equation is considered on the circle. In section 3, we prove results similar to those in
 section 1 but for the general $N\times N$ systems. The section 4
 contains the applications of these results to evolution in Hilbert
 spaces, e.g., the non-stationary Schr\"odinger equation. In the last section,
 we consider the most difficult case when the so-called
 ``gap condition" deteriorates in time. We handle only the
 short-range case, the general situation is far more difficult and will be considered
 elsewhere. The paper is concluded with Appendix which contains some
 well-known results we use in the main text. The first two sections
 and Appendix have mostly pedagogical value, the main results
 are in sections~3 and~5.

 In the text, we will use the following notations: for
 $f_{1(2)}\geq 0$, $f_1\lesssim f_2$ means there is a constant $C$
so that $f_1\leq Cf_2$. The symbols $e_j$ will denote the standard
 basis vectors in $\mathbb{C}^N$.

\section{The model case of $2\times 2$ system}

We start this section with the following model system of two ODEs.

\begin{equation}
X_t=i(\Lambda+V(t))X, \quad X(0)=I
\end{equation}
where
\begin{equation}
\Lambda=
\left[
\begin{array}{cc}
\lambda_1 & 0 \\
0 & \lambda_2
\end{array}
\right], \quad \lambda_{1(2)}\in \mathbb{R}
\end{equation}
and Hermitian $V(t)$
\[
V(t)=\left[
\begin{array}{cc}
v_{11}(t) & v_{12}(t) \\
\bar{v}_{12}(t) & v_{22}(t)
\end{array}
 \right]\to 0, \quad t\to\infty
\]
in some way to be specified below. Notice that $X(t)$ is unitary in
this setting.  The standard problem to address is, of course, the
long-time asymptotics of $X(t)$. In the case of strong decay, i.e.
when $V(t)\in L^1(\mathbb{R}^+)$, $X(t)$ has a limit as can be
easily seen from iteration of the corresponding integral equations.
Take
\[
X=UY
\]
where
\[
U=\left[
\begin{array}{cc}
\exp\left[i\lambda_1t+i\int\limits_0^tv_{11}(\tau)d\tau\right]  & 0 \\
0 & \exp\left[i\lambda_2t+i\int\limits_0^t v_{22}(\tau)d\tau
\right]
\end{array}
 \right]
\]
That reduces the problem to
\begin{equation}
Y_t=i\left[
\begin{array}{cc}
0 & q(t) \\
\bar{q}(t) & 0
\end{array}
 \right]Y,\quad Y(0)=I\label{model1}
\end{equation}
with
\[
q(t)=v_{12}(t)\exp\left[i(\lambda_2-\lambda_1)\tau+i\int\limits_0^t(v_{22}(\tau)-v_{11}(\tau))d\tau\right]
\]
and it decay as fast as $v_{12}$. Does $Y(t)$ have a limit as
$t\to\infty$? Clearly, for $q(t)\in L^1(\mathbb{R}^+)$ the answer is
yes but already for $q(t)\in L^p(\mathbb{R}^+), \quad p>1$ this is
not the case in general. Indeed, take $q(t)=i\epsilon$ on $t\in
[0,\epsilon^{-1}]$. Then, on that interval,
\[
Y(t)=\left[
\begin{array}{cc}
\cos(\epsilon t) & -\sin(\epsilon t) \\
\sin(\epsilon t) & \cos(\epsilon t)
\end{array}
 \right]
\]
a rotation by $\epsilon t$. Clearly,
$\|q\|_{L^1[0,\epsilon^{-1}]}=1$ but
$\|q\|_{L^p[0,\epsilon^{-1}]}=\epsilon^{1-p^{-1}}$. Therefore, by
taking $q(t)=\epsilon_n, \epsilon_n\to 0$ on consecutive intervals,
we can arrange for $q$ to be in $L^p(\mathbb{R}^+)$ but the solution
has no limit. In the meantime, we will see that it makes sense to
talk about the following problem: let $\lambda_1=0$, $\lambda_2=k$,
$v_{11}=v_{22}=0$ and $v_{12}=q(t)$. Thus,
\begin{equation}
X_t=i\left[
\begin{array}{cc}
0 & q(t) \\
\bar{q}(t) & k
\end{array}
\right]X,\quad X(0)=I \label{model31}
\end{equation}
and we will study the asymptotics for {\bf generic} frequency $k$.
As the discrete version of (\ref{model31}) one might consider the
following recursion
\begin{eqnarray*}
X_{n+1}(z)=\rho_n^{-1/2} \left[
\begin{array}{cc}
1 & \gamma_n \\
-\bar{\gamma}_n & 1
\end{array} \right] \cdot \left[
\begin{array}{cc}
z & 0\\
0 & 1
\end{array}
\right]X_n(z) \\
 X_0(z)=I_{2\times 2}, \quad |z|=1, \quad
\rho_n=1+|\gamma_n|^2
\end{eqnarray*}
This is different from the recursion for polynomials orthogonal on
the circle \cite{Sim1} only by sign "$-$" in front of $\gamma_n$ and
by slightly different formula for $\rho_n$.

 {\bf Remark 1}. The simple calculation shows that
\[\displaystyle
X(t,k)= \left[
\begin{array}{cc}
x_{11}(t,k) & x_{12}(t,k) \\
x_{21}(t,k) & x_{22}(t,k)
\end{array}
\right] =\left[
\begin{array}{cc}
x_{11}(t,k) & x_{12}(t,k) \\
-e^{ikt}\overline{x_{12}(t,k)} & e^{ikt}\overline{x_{11}(t,k)}
\end{array}
\right]
\]
for real $k$.
\bigskip

{\bf Open problem.} Is it true that for $q(t)\in L^2(\mathbb{R}^+)$
and Lebesgue a.e. $k\in \mathbb{R}$, the limit
$\lim\limits_{t\to\infty} X(t,k)$ exists? If so, how do the points
of convergence correspond to frequencies $k$ for which
$[Mq](k)<\infty$, where $Mq$ is the Carleson-Hunt maximal function?

The similar problem is known for the Krein systems \cite{Den1}
\[
Z_t=\left[
\begin{array}{cc}
ik & \bar{q}(t) \\
q(t) & 0
\end{array}
\right]Z, \quad Z(0)=I, \quad k\in \mathbb{R}
\]
In this case, though, the solution $Z$ is $J$-unitary, with
\[
J=\left[
\begin{array}{cc}
1 & 0 \\
0 & -1
\end{array}
\right]
\]

We will consider (\ref{model31}) and will prove a simple result
which has an analog in the theory of Krein systems (the so-called
Szeg\H{o} case) which gives a somewhat weaker type of convergence.
That will be a warm-up for a later discussion of the general case.
 Consider the first column
$(x_{11}(t,k),x_{21}(t,k))^t$ of $X(t,k)$. The functions
$x_{11(21)}(t,k)$ are holomorphic in $k$ and have the following
properties for any fixed $t$.

\begin{lemma}
For any locally integrable $q(t)$, we have

\begin{itemize}
\item[(a)]
 $\displaystyle |x_{11}(t,k)|^2+|x_{21}(t,k)|^2=1-2\Im
k\int\limits_0^t |x_{21}(\tau,k)|^2d\tau,\quad  k\in \mathbb{C} $

\item[(b)]
$|x_{11(21)}(t,k)|\leq 1, \quad k\in \overline{\mathbb{C}^+}$, and
$\|x_{21}(t,k)\|_2\leq (2\Im k)^{-1/2}$ for $k\in \mathbb{C}^+$.

\end{itemize}
\end{lemma}
\begin{proof}
This is a direct corollary of the differential equations.
\end{proof}

Assuming, in addition, that $q(t)$ is square summable, we get

\begin{lemma} If $q(t)\in L^2(\mathbb{R}^+)$, then
\begin{itemize}
\item[(a)]
$x_{11}(t,k)\to x_{11}(\infty,k), \quad x_{21}(t,k)\to 0$
 uniformly on compacts in $\mathbb{C}^+$

\item[(b)]
 $|x_{11}(t,k)|\geq 1-(\Im
k)^{-1}\|q\|_2^2, \quad \|x_{21}(t,k)\|_2\leq \|q\|_2(\Im
k)^{-1},\quad k\in \mathbb{C}^+$

\item[(c)] The function $x_{11}(\infty,k)$ is nonzero function in the unit ball in
$H^\infty(\mathbb{C}^+)$ and $x_{11}(k,t)\to x_{11}(\infty,k)$ in
the weak-$\ast$ sense on $\mathbb{R}$

\item[(d)] The following estimates hold
\begin{eqnarray}
\int\limits_\mathbb{R} \ln |x_{11}(t,k)|dk\geq -\pi\int\limits_0^t
|q(\tau)|^2d\tau \label{first}\hspace{3cm}\\
 0\geq \int\limits_\mathbb{R} (|x_{11}(\infty,k)|-1)dk\geq
\int\limits_\mathbb{R} \ln |x_{11}(\infty,k)|dk\geq
-\pi\int\limits_0^\infty |q(\tau)|^2d\tau \label{second}
\end{eqnarray}

\item[(e)] We have
\begin{equation}
\left[1-x_{11}(\infty,k)\right]_{L^{1,w}(\mathbb{R})}\lesssim
\|q\|_2^2, \quad \|1-x_{11}(\infty,k)\|_{L^{p}(\mathbb{R})}\leq
C(p)\|q\|_2^{2/p}, \quad 1<p<\infty \label{este1}
\end{equation}

\end{itemize}
\label{lemka}
\end{lemma}
\begin{proof}
If $x_{21}=e^{ikt}y$, then we have the following integral equations
\begin{eqnarray*}
x_{11}(t,k)=1+i\int\limits_0^t q(\tau)e^{ik\tau}y(\tau,k)d\tau,\quad
y(t,k)=i\int\limits_0^t \bar{q}(\tau)e^{-ik\tau}x_{11}(\tau,k)d\tau
\end{eqnarray*}
Therefore,
\begin{equation}
x_{11}(t,k)=1-\int\limits_0^t
q(\tau_1)e^{ik\tau_1}\int\limits_0^{\tau_1} \bar{q}(\tau_2)
e^{-ik\tau_2} x_{11}(\tau_2,k)d\tau_2 d\tau_1 \label{rep1}
\end{equation}
Due to the estimate $|x_{11}(t,k)|\leq 1$ and Young's inequality,
the $L^1$ norm of the integrand in $\tau_1$ is smaller than
$\|q\|_2^2 (\Im k)^{-1}$. That implies the uniform convergence on
compacts of $\mathbb{C}^+$ and the bound from below for $x_{11}$.
The representation
\[
x_{21}(t,k)=ie^{ikt}\int\limits_0^t
\bar{q}(\tau)e^{-ik\tau}x_{11}(\tau,k)d\tau
\]
gives uniform convergence to zero for $x_{21}(t,k)$ and the estimate
on its $L^2$ norm. The function $x_{11}(\infty,k)$ is obviously in
the unit ball of $H^\infty(\mathbb{C}^+)$ and is nonzero due to the
estimate from below. Since $x_{11}(t,k)(k+i)^{-1},
x_{11}(\infty,k)(k+i)^{-1}$ are both uniformly bounded in
$H^2(\mathbb{C}^+)$ and
\[
x_{11}(t,k)(k+i)^{-1}\to x_{11}(k)(k+i)^{-1}
\]
as $t\to\infty$ uniformly on the compacts in $\mathbb{C}^+$, we have
\[
x_{11}(t,k)(k+i)^{-1}\to x_{11}(\infty,k)(k+i)^{-1}
\]
 weakly in
$H^2(\mathbb{C}^+)$. Thus, we have the weak-$\ast$ convergence of
$x_{11}(t,k)$ to $x_{11}(\infty,k)$ on $\mathbb{R}$.

Let us prove estimates in (d) now. Iterate (\ref{rep1}) once and
take the first term (the second allows stronger estimate)
\[
\int\limits_0^t q(\tau_1)e^{ik\tau_1}\int\limits_0^{\tau_1}
\bar{q}(\tau_2) e^{-ik\tau_2}d\tau_2 d\tau_1=i\int
|\hat{q}_t(\omega)|^2(\omega+k)^{-1}d\omega
\]
by Plancherel, where $\hat{q}_t(\omega)$ is the Fourier transform of
$q(\tau)\cdot \chi_{[0,t]}(\tau)$.
 Therefore, we have
\[
x_{11}(t,k)=1-ik^{-1}\int\limits_0^t
|q(\tau)|^2d\tau+\overline{o}((\Im k)^{-1})
\]
as $\Im k\to +\infty$. The function $\ln |x_{11}(t,k)|$ is
subharmonic in $\mathbb{C}^+$ so
\[
\pi^{-1}\int \frac{y \ln |x_{11}(t,s)|}{(s-x)^2+y^2}ds\geq \ln
|x_{11}(t,k)|, \quad k=x+iy
\]
Since $\ln |x_{11}(t,s)|\leq 0$, we can take $k=iy, y\to\infty$ and
compare the first terms in asymptotics of l.h.s. and r.h.s. to get
(\ref{first}). The estimates for $x_{11}(\infty,k)$ are deduced
similarly.

Let us prove (e). Writing $x_{11}(\infty,k)=1+h(k)$ yields $2\Re
h\leq -|h|^2$ since $|x_{11}|\leq 1$ and
\[
\Re h(iy)=\frac{y}{\pi} \int \frac{\Re h(t)}{t^2+y^2} dt
\]
Taking $y\to +\infty$, we get

\[
\int_\mathbb{R} \Re h(t)dt=-\pi\int_0^\infty |q(\tau)|^2d\tau, \quad
\Re h(t)\leq 0, \quad |h(t)|\leq 2
\]
(this we might call the ``trace formula" for our evolution).
Therefore, we have
\[
\left[h\right]_{L^{1,w}(\mathbb{R})}\lesssim \|q\|_2^2
\]
and
\[
\|h\|_{L^{p}}(\mathbb{R})\leq C(p)\|q\|_2^{2/p},\quad 1<p<\infty
\]
by interpolation.
\end{proof}

In a similar way, one can show that $x_{12}(t,k)$ has a limit in
$\mathbb{C}^+$ which we will denote by $x_{12}(\infty,k)$. This, of
course, implies that $x_{12}(t,k)\to x_{12}(\infty,k)$ in the
weak-$\ast$ sense on the real line. In the next lemma, we will
improve the convergence result.
\begin{lemma}
If $q(t)\in L^2(\mathbb{R}^+)$, then for any $j=1,2$,

\[
\int\limits_{\mathbb{R}}|x_{1j}(t,k)- x_{1j}(\infty,k)|^2 dk\to 0,
\quad t\to\infty
\]\label{lem}
\end{lemma}
\begin{proof}
Take arbitrary $t_1<t_2$ and notice that we have the following
semigroup property
\[
X(t_1,t_2,k)X(t_1,k)=X(t_2,k)
\]
Hence,

\[
x_{1j}(t_2,k)-x_{1j}(t_1,k)=(x_{11}(t_1,t_2,k)-1)x_{1j}(t_1,k)+x_{12}(t_1,t_2,k)x_{2j}(t_1,k)
\]
By (\ref{este1}),
\[
\|x_{11}(t_1,t_2,k)-1\|_2\to 0
\]
as $t_{1(2)}\to \infty$ and
\[
|x_{12}(t_1,t_2,k)|^2=1-|x_{11}(t_1,t_2,k)|^2\leq -2\ln
|x_{11}(t_1,t_2,k)|
\]
So, $\|x_{12}(t_1,t_2,k)\|_2\to 0$ by (\ref{second}). Thus,
$x_{1j}(t,k)$ is Cauchy in $L^2(\mathbb{R})$ and must have the limit
which will coincide with weak-$\ast$ limit $x_{1j}(\infty,k)$.

\end{proof}
{\bf Remark 2.}  By Remark 1, $e^{-ikt}x_{21}(t,k)$ and
$e^{-ikt}x_{22}(t,k)-1$ have limits in $L^2(\mathbb{R})$. They are
in fact related to boundary values of $H^\infty(\mathbb{C}^-)$
functions. Notice that
\[
|x_{11}(\infty,k)|^2+|x_{21}(\infty,k)|^2=1, \quad {\rm for \quad
a.e.} \quad k\in \mathbb{R}
\]
as follows from the a.e. convergence over some subsequence.

\section{The transport equation on the circle}
Another model very important for us is the transport equation.
Consider
\begin{equation}
u_t=ku_x+q(t,x)u \label{transport}
\end{equation}
with the simplest initial condition $u(0,x)=1$. We can either say
that $x\in \mathbb{T}$ or $x\in \mathbb{R}$ but all functions are
$2\pi$--periodic in $x$.

Notice that on the Fourier side, this equation is
\[
\hat{u}_t=-ikD\hat{u}+\hat{q}(t)\ast \hat{u}, \quad
\hat{u}(0,n)=\delta_0
\]
where $D$ is diagonal: $(Dg)_n=ng(n)$ in $\ell^2(\mathbb{Z})$. We
will work under the assumption that
\begin{equation}
\hat{q}(t,0)=\int\limits_{\mathbb{T}} q(t,x)dx=0 \label{pred}
\end{equation}
There is no any loss of generality since we can always satisfy this
condition by subtracting $\hat{q}(t,0)I$ which corresponds to
unimodular factor for $u(t,x)$.

{\bf Remark 3.} Let $P_{\{0,1\}}$ be the Fourier projection onto the
zeroth and the first modes, $P_{\{0,1\}}^c=I-P_{\{0,1\}}$. Then,
\[
(P_{\{0,1\}}u)_t=k(P_{\{0,1\}}u)_x+(P_{\{0,1\}}qP_{\{0,1\}})(P_{\{0,1\}}u)+
P_{\{0,1\}}qP_{\{0,1\}}^cu
\]
Notice that if one drops the third term in the r.h.s., then the
equation becomes equivalent to the model considered in the previous
section.

Of course, the solution to (\ref{transport}) can be written
explicitly. After periodization, we have
\[
u(t,x,k)=\exp\left[ \int\limits_0^t q(\tau, x+k(t-\tau))d\tau\right]
\]
If, e.g., $ q, q_x\in C(\mathbb{R}^+,\mathbb{T})$, then this is the
classical solution. The meaningful question is to study the
asymptotics of $u(t,x-kt)$ or, rather,
\[
\phi(t,x,k)=\int\limits_0^t q(\tau, x-k\tau)d\tau
\]
On the Fourier side (with respect to $x$ coordinate),
\[
\hat\phi(t,n,k)=\int\limits_0^t e^{in\tau k} \hat{q}(\tau,n)d\tau
\]
and
\[
\hat\phi(t,0,k)=0, \quad t>0
\]
due to assumption (\ref{pred}). We have the following
\begin{lemma} If $q(t,x)\in L^2(\mathbb{R}^+,\mathbb{T})$ and
$\displaystyle \int_\mathbb{T} q(t,x)dx=0$, then
\begin{itemize}
\item[1.] There is $\phi(\infty,x,k)$ such that
\[
\int\limits_{\mathbb{R}}
\|\phi(t,x,k)-\phi(\infty,x,k)\|^2_{H^{1/2}(\mathbb{T})}dk\to 0
\]
as $t\to\infty$.

\item[2.] For a.e. $k$,
\[
\|\phi(t,x,k)-\phi(\infty,x,k)\|_{H^{1/2}(\mathbb{T})}\to 0 \quad
t\to\infty
\]

\end{itemize}
\end{lemma}

\begin{proof}
The proof is a trivial calculation. We have
\begin{eqnarray*}
\int\limits_\mathbb{R} \|\phi(t,x,k)\|_{H^{1/2}(\mathbb{T})}^2dk=\hspace{4cm}
\\
=\sum_{n\neq 0} |n|\int\limits_{\mathbb{R}} \left| \int\limits_0^t
e^{ink\tau}
\hat{q}(\tau,n)d\tau\right|^2dk=(2\pi)^2\int\limits_\mathbb{T}dx\int\limits_0^t
|q(\tau,x)|^2d\tau
\end{eqnarray*}
by Plancherel. If $\phi(\infty,x,k)$ is defined as function on
$\mathbb{T}\times \mathbb{R}$ having Fourier coefficients
\[
\int\limits_0^\infty e^{in\tau k}\hat{q}(\tau,n)d\tau
\]
then we easily have the first statement of the lemma.

For the second part, it is sufficient to show that for a.e. $k$ we
have
\[
h(t_1,t_2,k)=\sum\limits_{n\in \mathbb{Z}}
|n|\left|\int\limits_{t_1}^{t_2} e^{in\tau k}\hat{q}(\tau,n)d\tau
\right|^2\to 0, \quad t_{1(2)}\to+\infty
\]
since then the convergence will follow from the Cauchy criterion.
Let us introduce
\[
h_m(t,k)=\sum\limits_{n\in \mathbb{Z}} |n|\left(\sup\limits_{\tau_1,
\tau_2>t}\left|\int\limits_{\tau_1}^{\tau_2} e^{in\tau
k}\hat{q}(\tau,n)d\tau \right|^2\right)
\]
We have
\[
0\leq h(t_1,t_2,k)\leq h_m(t,k), \quad t<t_1, t_2
\]
and $h_m(t,k)$ is decreasing in $t$ (if it exists). Moreover,
\[
\int\limits_\mathbb{R} h_m(t,k) dk=\sum\limits_{n\in \mathbb{Z}}
|n|\int\limits_{-\infty}^\infty\left(\sup\limits_{\tau_1,
\tau_2>t}\left|\int\limits_{\tau_1}^{\tau_2} e^{in\tau
k}\hat{q}(\tau,n)d\tau \right|^2\right)dk
\]
\begin{eqnarray*}
=\sum\limits_{n\neq 0}
\int\limits_{-\infty}^\infty\left(\sup\limits_{\tau_1,
\tau_2>t}\left|\int\limits_{\tau_1}^{\tau_2} e^{i\tau\xi
}\hat{q}(\tau,n)d\tau \right|^2\right)d\xi\lesssim
\sum\limits_{n\neq 0} \int\limits_t^\infty |\hat{q}(\tau,n)|^2d\tau\\=
2\pi\|q\|^2_{L^2([t,\infty)\times\mathbb{T})}\to 0
\end{eqnarray*}
Here we used the standard Carleson estimate for the maximal function
\cite{Carl}. Since $h_m(t,k)$ is monotonic, we have $h_m(t,k)\to 0$
for a.e. $k$.
\end{proof}

We get the following simple corollary
\begin{corollary}
If $q(t,x)\in L^2(\mathbb{R}^+,\mathbb{T})$ and $\displaystyle
\int_\mathbb{T} q(t,x)dx=0$, then for a.e. $k$ we have
$u(t,x-kt,k)\to \nu(x,k)$ in the following sense
\[
\|u(t,x-kt,k)-\nu(x,k)\|^2_{L^2(\mathbb{T})} \to 0, \quad t\to
\infty
\]
If we also have $q(t,x)\in i\mathbb{R}$, then for a.e. $k$ there is
$\nu(x,k)\in H^{1/2}(\mathbb{T})$ such that
\[
\|u(t,x-kt,k)-\nu(x,k)\|^2_{H^{1/2}(\mathbb{T})} \to 0, \quad t\to
\infty
\]
\end{corollary}
\begin{proof}
We know for a.e. $k$ the limit $\phi(\infty,x,k)$ exists as a
function in  $H^{1/2}(\mathbb{T})$. All other statements follow from
Lemmas \ref{al1}, \ref{al2} in Appendix.
\end{proof}
{\bf Remark 4.} Similarly, one can show
\begin{equation}
\int\limits_\mathbb{R}
\|u(t,x-tk,k)-\nu(x,k)\|_{H^{1/2}(\mathbb{T})}^2dk\to 0,
\end{equation}
provided that $q\in i\mathbb{R}$.  It follows from lemma \ref{al2},
estimates on the maximal function, and dominated convergence
theorem. Moreover, for a.e. $k$, all Fourier coefficients of
$u(t,x-kt,k)$ converge as $t\to\infty$ regardless of whether $q$ is
purely imaginary or not.
\bigskip

{\bf Remark 5.}  We considered the simplest case of initial data,
i.e. $u(0,x,k)=1$. The general case $u(0,x,k)=f(x)$ is almost
identical due to multiplicative structure of the problem. If the
potential is square summable and purely imaginary, then we have the
full measure set of $k$ (that depends only on $q$) for which the
equation is globally well-posed for $f$ in the Krein algebra
$L^\infty(\mathbb{T})\cap H^{1/2}(\mathbb{T})$ \cite{BS}. We also
have the stability and the asymptotics at infinity. \bigskip\bigskip

There is an instructive case $q(t,x)=2q(t)\cos x$ with $q(t)$ --
purely imaginary square summable on $\mathbb{R}$. In this situation,

\begin{equation}
\nu(x,k)=\exp\left[\left(
e^{-ix}\hat{q}(k)-e^{ix}\overline{\hat{q}(k)}\right)\right]
\label{inst}
\end{equation}
where
\[
\hat{q}(k)=\int\limits_0^\infty e^{ik\tau}q(\tau)d\tau
\]
Notice that for a.e. $k$ the function $\nu(x,k)$ is infinitely
smooth. Moreover, expanding into the Taylor series,
\[
\int_\mathbb{T} \nu(x,k)dx=H(|\hat{q}(k)|)
\]
with
\[
H(z)=\sum\limits_{n=0}^\infty \frac{(-z^2)^{n}}{(n!)^2}=J_0(2z)
\]
Notice that since the Bessel function $J_0(z)$ has  positive zeroes
(\cite{as}, Chapter 9), it is possible to choose $q$ such that
\[
\hat\nu(0,k)=\int_\mathbb{T} \nu(x,k)dx=0
\]
on arbitrary interval $k\in I$ which means there is no hope to get
\begin{equation}
\int\limits_I \ln |\hat \nu(0,k)|dk>-\infty \label{never}
\end{equation}\bigskip\bigskip

Consider the case when the transport equation is given on the
cylinder of large size $2\pi h$
\[u_t=ku_x+q(t,x)u, \quad u(0,x)=1,
\]
and $u$ is $h$--periodic in $x$, $q$ is purely imaginary. Scaling in
$x$ gives
\[
\psi_t=kh^{-1}\psi_\theta+\tilde{q}(t,\theta)\psi, \quad
\psi(0,\theta)=1
\]
where $\psi(t,\theta,k)=u(t,h\theta,k)$,
$\tilde{q}(t,\theta)=q(t,h\theta)$. For the new differential
operator, $ih^{-1}\partial_x$, the gaps in the spectrum are of the
size $h^{-1}$ but, nevertheless, we have

\begin{equation}
h^{-1}\int\limits_\mathbb{R}
\|u(T,hx,k)\|^2_{H^{1/2}(\mathbb{T})}dk\lesssim
h^{-1}\left(1+\int\limits_0^T \int\limits_0^h q^2(t,x)dxdt\right)
\label{evid}
\end{equation}
by scaling. The r.h.s. measures the $L^2$ norm in time of the space
averages of $q$. If it is bounded, then $H^{1/2}$ norm of
$u(T,x,k)$, when averaged over $(0,h)$, is bounded for most $k$. We
expect this phenomenon for general situation when the gap condition
deteriorates.

The calculations presented in this section can be easily carried out
for the case when $q$ is more regular, e.g. $q\in
L^2(\mathbb{R}^+,H^{1/2}(\mathbb{T}))$. That will lead to better
regularity of the solution.

\section{ The model case of $N\times N$ system}
In this section, we consider the following evolution
\[
X_t=i(k\Lambda+V(t))X, \quad X(0)=I,\quad V^*=V
\]
and
\[
\Lambda=\left[
\begin{array}{cccc}
\lambda_1 & 0 & \ldots  &0\\
0 & \lambda_2 & \ldots &\ldots\\
\ldots & \ldots & \ldots & \ldots\\ 0 & 0 & \ldots & \lambda_N
\end{array}
\right]
\]
with $0=\lambda_1< \lambda_2< \ldots<\lambda_N$. Sometimes we will
allow the eigenvalues to degenerate, that will require more careful
analysis. Denote $\delta_j=\lambda_{j+1}-\lambda_j, \quad j=1,
\ldots , N-1$. We will also assume that $V(t)$ is locally integrable
on $\mathbb{R}^+$ and that $V_{jj}(t)=0$ for all $j$. The last
assumption can be made without loss of generalization. It is obvious
that $X(t,k)=\{x_{mn}(t,k), \, 1\leq m,n\leq N\}$ is unitary for
real $k$. For general $k$, the following lemma holds true.
\begin{lemma}
For any $V\in L^1_{\rm loc}(\mathbb{R}^+)$, we have

\[
|X(t,k)|^2+2\Im k \int\limits_0^t X^*(\tau,k)\Lambda
X(\tau,k)d\tau=I,\quad k\in \mathbb{C}
\]
and
\begin{equation}
0\leq \int\limits_0^\infty X^*(\tau,k)\Lambda X(\tau,k)d\tau\leq
(2\Im k)^{-1}, \quad |X(t,k)|\leq I,\quad \quad \Im k>0
\label{est-m1}
\end{equation}
Moreover, $|X(t,k)|$ and $|X(t,k)|^2$ decay monotonically in $t$ for
$k\in \mathbb{C}^+$.
\end{lemma}
\begin{proof}
The proof is a trivial corollary from the differential equation
itself and monotonicity of the square root.
\end{proof}

\begin{lemma} \label{lemochka}
Assume that $q(t)=\|V(t)e_1\|$ belongs to $L^2(\mathbb{R}^+)$. Fix
any $f\in \mathbb{C}^N$ with $\|f\|=1$.
\begin{itemize}
\item[(a)] We have $X(t,k)f\to \pi_f(k)e_1$, as $t\to \infty$ uniformly on
compacts in $\mathbb{C}^+$.
\item[(b)] For $k\in \mathbb{C}^+$, $|\langle X(t,k)f, e_1\rangle|
\geq |\langle f,e_1\rangle|-
(2\lambda_1\Im k)^{-1/2}\cdot\|q\|_2$ and
\[
\left(\int\limits_0^\infty \|P_1^c
X(t,k)f\|^2dt\right)^{1/2}\lesssim (\lambda_1\Im
k)^{-1}\|q\|_2+(\lambda_1 \Im k)^{-1/2}\|P_1^cf\|
\]

\item[(c)] $\langle
X(t,k)f, e_1\rangle$ converges to $\pi_f(k)$ on $\mathbb{R}$ in the
weak-$\ast$ sense. If $\langle f,e_1\rangle\neq 0$, then $\pi_f(k)$
is nonzero function in  the unit ball in $H^\infty(\mathbb{C}^+)$
and so

\[
\int\limits_{-\infty}^\infty \frac{\ln |\pi_f(k)|}{k^2+1}dk>-\infty
\]

\item[(d)] For $x_{11}(t,k)$, we have $x_{11}(t,k)\to
x_{11}(\infty,k)$ uniformly over compacts in $\mathbb{C}^+$.
Moreover,

\begin{equation}
 \int\limits_\mathbb{R} \ln
|x_{11}(\infty,k)|dk\gtrsim -\lambda_1^{-1}\|q\|_2^2
 \label{second2}
\end{equation}
and
\begin{equation}
\left[1-x_{11}(\infty,k)\right]_{L^{1,w}(\mathbb{R})}\lesssim
\|q\|_2^2, \quad \|1-x_{11}(\infty,k)\|_{L^{p}(\mathbb{R})}\leq
C(p)\|q\|_2^{2/p}, \quad 1<p<\infty \label{e15}
\end{equation}

\end{itemize}
\end{lemma}
\begin{proof}
Denote $u(t,k)=X(t,k)f$. It is entire in $k$. We have $u_t=ik\Lambda
u+iVu,$ $u(0,k)=f$. Then,
\begin{equation}
\langle u(t,k),e_1\rangle=\langle f,e_1\rangle +i\int\limits_0^t
\langle u(\tau,k),V(\tau)e_1\rangle d\tau \label{rep2}
\end{equation}
Since $V_{11}(t)=0$, we have $\langle
u(\tau,k),V(\tau)e_1\rangle=\langle
P_1^cu(\tau,k),V(\tau)e_1\rangle$. From (\ref{est-m1}), we have
\[
\int\limits_0^\infty \|P_1^cu(\tau,k)\|^2d\tau\leq (2\lambda_1\Im
k)^{-1}
\]
Thus $\langle u(\tau,k),V(\tau)e_1\rangle\in L^1(\mathbb{R}^+)$ by
Cauchy-Schwarz and that proves convergence of $\langle
u(t,k),e_1\rangle$ to some $\pi_f(k)$ and
\begin{equation}
|\pi_f(k)-\langle f,e_1\rangle|\leq (2\lambda_1\Im k)^{-1/2} \|q\|_2
\label{estbn}
\end{equation}

Consider $\psi(t)=P_1^cu$. We have
\[
\psi_t=ik P_1^c\Lambda P_1^c \psi +iP_1^cVP_1^c\psi+iP_1^cVP_1u,
\quad \psi(0)=P_1^cf
\]
If
\[
l(t,k)=iP_1^cVP_1u=i\langle u,e_1\rangle P_1^cVe_1
\]
then
\[
\int_0^\infty \|l(\tau,k)\|^2d\tau\leq \|q\|_2^2
\]
since $|\langle u,e_1\rangle |\leq 1$. Thus,

\[
\frac{d}{dt} \left(\|\psi\|_2^2\right)\leq -2\lambda_1\Im k
\|\psi\|_2^2+2\|l\|\cdot\|\psi\|, \quad \|\psi(0,k)\|\leq 1
\]
and we have
\[
\| \psi(t,k)\|\lesssim e^{-\alpha t}\|P_1^cf\|+\int\limits_0^t
e^{-\alpha(t-\tau)}\|l(\tau)\|d\tau, \quad \alpha=\lambda_1\Im k
\]
So,
\begin{equation}
\left(\int\limits_0^\infty\|\psi(\tau,k)\|^2d\tau\right)^{1/2}\lesssim
(\lambda_1\Im k)^{-1}\|q\|_2+(\lambda_1\Im k)^{-1/2}\|P_1^cf\|
\label{ost}
\end{equation}
and $\|\psi(t,k)\|\to 0$ uniformly on compacts in $\mathbb{C}^+$.
That proves (a) through (b). The properties of $\pi_f(k)$ stated in
(c) follow from the mean-value inequality for subharmonic function
$\ln |\pi_f(k)|$ and (\ref{estbn}).

 If $f=e_1$, then (\ref{rep2}) and
(\ref{ost})
 yield

\[
|x_{11}(t,iy)-1|\lesssim \lambda_1^{-1}y^{-1}\|q\|_2^2
\]
The proof of (d) repeats the arguments in lemma \ref{lemka}.
\end{proof}
As a simple corollary of (c), we get existence of the weak-$\ast$
limits for $x_{1j}(t,k)$ on the real line ($j=2,\ldots, N$). Denote
them by $x_{1j}(\infty,k)$. The next lemma gives a stronger
convergence result and is an analog of lemma \ref{lem}
\begin{lemma}Assume that $q(t)=\|V(t)e_1\|$ belongs to $L^2(\mathbb{R}^+)$. Then,
\[
\int\limits_{-\infty}^\infty |x_{1j}(t,k)-x_{1j}(\infty,k)|^2dk\to
0, \quad j=1,\ldots, N
\]
as $t\to\infty$.
\end{lemma}
\begin{proof}
For any $t_1<t_2$ the semigroup property yields
\[
x_{1j}(t_2,k)=\sum\limits_{m=1}^N x_{1m}(t_1,t_2,k)x_{mj}(t_1,k)
\]
where $X(t_1,t_2,k)$ has matrix elements $\{x_{ij}(t_1,t_2,k)\}$.
Therefore,
\[
x_{1j}(t_2,k)-x_{1j}(t_1,k)=\sum\limits_{m>1}
x_{1m}(t_1,t_2,k)x_{mj}(t_1,k)+(x_{11}(t_1,t_2,k)-1)x_{1j}(t_1,k)
\]
\[
=I_1+I_2
\]
By (\ref{e15}), we have
\[
\|I_2(k)\|_{L^2(\mathbb{R})}\to 0
\]
as $t_{1(2)}\to \infty$. The Cauchy-Schwarz and unitarity of $X$
give
\[
|I_1(k)|^2\leq 1-|x_{11}(t_1,t_2,k)|^2\leq -2\ln |x_{11}(t_1,t_2,k)|
\]
and thus
\[
\|I_1(k)\|_{L^2(\mathbb{R})}\to 0
\]
by (\ref{second2}). Therefore, $x_{1j}(t,k)$ is Cauchy in
$L^2(\mathbb{R})$ and must have a limit equal to the the weak-$\ast$
limit $x_{1j}(\infty,k)$.
\end{proof}

For fixed $f$, we have
$\pi_f=x_{11}(\infty,k)f_1+\ldots+x_{1N}(\infty,k)f_N$ and so

\begin{corollary}
If $q(t)=\|V(t)e_1\|\in L^2(\mathbb{R}^+)$, then
\[
\int\limits_\mathbb{R}|P_1X(t,k)f-\pi_f(k)|^2dk\to 0
\]
for any $f\in \mathbb{C}^N$.
\end{corollary}
This result is somewhat surprising since we do not assume anything
about $P_1^cVP_1^c$ except local integrability.\bigskip

Now, we are going to prove results on convergence of all elements of
the matrix $X$ and need to assume more on $V$. Let $\|V\|\in
L^2_{\rm loc}(\mathbb{R}^+)$ and $T$ is a fixed positive constant.
We start with the following simple observation. Fix $1< j< N$ and
consider vector $u(t)$ (it will be different for different $j$ but
we  suppress this dependence for shorthand) which solves
\begin{equation}
\frac{d}{dt} u(t)=i(k\Lambda+V)u, \quad 0<t<T \label{ev1}
\end{equation}
and satisfies the following boundary conditions. Let $a(t)$ denote
the vector containing the first $j-1$ components of $u$, $b(t)$ is
the $j$--th component of $u$, and $c(t)$ contains the $j+1,\ldots,
N$ components of $u$. Then, we require that $c(0)=0$, $b(0)=1$, and
$a(T)=0$. This solution does not have to exist, but for $\Im k$
large enough or small $V$  it does, it is unique, and it allows two
different representations. One of them is through $X$. Let
$X_j=P_{1\leq k\leq j}XP_{1\leq k\leq j}$, where $P_{1\leq k\leq j}$
is the projection onto the first $j$ coordinates. Then, assuming
that $u$ exists,
\begin{equation}
X_j(T)(a(0),1)^t=b(T)(0,\ldots,0,1)^t \label{ur}
\end{equation}
By the Laplace theorem for determinants, we have
\begin{equation}
\Delta_{j-1}b(T)=\Delta_{j} \label{vgw}
\end{equation}
where $\Delta_j=\det X_j$. Provided that $u_l$ exists for any
$l=1,\ldots, j$, iteration yields
\begin{equation}
\Delta_j(T,k)=b_1(T,k)\cdot\ldots\cdot b_j(T,k) \label{iter}
\end{equation}
The $b_1(t,k)$ can be identified with $x_{11}(t,k)$.

 The existence of $u(t,k)$ for large $\Im k$ and its analytical properties
follow from the standard asymptotical method for systems of ODEs
close to diagonal. We can write (\ref{ev1}) as
\[
\left\{\begin{array}{ccl}
a'&=&(ik\Lambda_a+Q_{11})a+ Q_{12}b+Q_{13}c\\
b'&=&Q_{21}a+ik\lambda_jb+Q_{23}c\\
c'&=&Q_{31}a+Q_{32}b+(ik\Lambda_c+Q_{33})c\\
\end{array}\right.
\]
where $Q_{nl}$ are the corresponding blocks of $iV$ and
\[
\Lambda_a=P_{1\leq n\leq  j-1}\Lambda P_{1\leq n\leq  j-1},\quad \Lambda_c=P_{j+1\leq n\leq N}\Lambda P_{j+1\leq n\leq N}
\]

Let $U_1$ and $U_2$ be solutions to the following Cauchy problems
\begin{eqnarray*}
\frac{d}{dt} U_1(\tau,t,k)=(ik\Lambda_a+Q_{11})U_1(\tau,t,k), \quad
U_1(\tau,\tau,k)=I
\\
\frac{d}{dt} U_2(\tau,t,k)=(ik\Lambda_c+Q_{33})U_2(\tau,t,k), \quad
U_2(\tau,\tau,k)=I
\end{eqnarray*}
Since $Q_{11}$ and $Q_{33}$ are antisymmetric, we have the following
obvious estimates
\begin{equation}
\|U_1(\tau,t,k)\|\leq \exp(\lambda_{j-1}\cdot \Im k\cdot (\tau-t)),
\quad t<\tau \label{est3}
\end{equation}
\begin{equation}
\|U_2(\tau,t,k)\|\leq \exp(-\lambda_{j+1}\cdot\Im k\cdot (t-\tau)),
\quad \tau<t \label{est4}
\end{equation}

The integral equations for the boundary conditions specified are
\begin{equation}
\left\{
\begin{array}{l}
\displaystyle a(t)=-\int_t^T U_1(\tau,t,k)\left[
Q_{12}(\tau)b(\tau)+Q_{13}(\tau)c(\tau)\right]d\tau
\\
\displaystyle b(t)=\exp(ik\lambda_jt)+\int_0^t
\exp(ik\lambda_j(t-\tau))\left[
Q_{21}(\tau)a(\tau)+Q_{23}(\tau)c(\tau)\right]d\tau
\\
\displaystyle c(t)=\int_0^t U_2(\tau,t,k)\left[
Q_{31}(\tau)a(\tau)+Q_{32}(\tau)b(\tau)\right]d\tau
\end{array}\right.\label{inteq}
\end{equation}
Consider $\tilde{u}=\exp(-ik\lambda_jt)u$. Then, for $\tilde{u}$, we
have the operator equation

\[
\tilde{u}(t)=f(t)+D\tilde{u}
\]
where $f(t)=e_j$ and $D$ is the corresponding integral operator. If
$y=\Im k>>1$,  then $D^2$ is contraction in the ball of radius $1$
in the space $\cal B$, where

\[
\cal B=\cal L\times\ldots\times \cal L\times L^\infty (0,T)\times
\cal L\times \ldots \times \cal L
\]
and $\cal L$ is the space with the norm
$\|\cdot\|_\infty+\|\cdot\|_2$. In fact, by (\ref{est3}),
(\ref{est4}), Cauchy-Schwarz and Young inequalities,
\[
\|(D)_{1(3)}\|_\infty\lesssim \|Q\|_2 y^{-1/2}, \|(D)_{1(3)}\|_2\lesssim \|Q\|_2 y^{-1}, \|(D)_2\|\lesssim \|Q\|_2
\]
\[
\|(D^2)_{1(3)}\|_\infty\lesssim \|Q\|_2^2 y^{-1/2},
\|(D^2)_{1(3)}\|_2\lesssim \|Q\|_2^2 y^{-1}, \|(D^2)_2\|\lesssim
\|Q\|_2^2y^{-1}
\]
Thus, there is a unique solution which belongs to $\cal B$. Assuming
$T=\infty$ and $\|V\|\in L^2(\mathbb{R}^+)$, we have
$u(t)=\exp(i\lambda_jkt)(b(\infty,k)e_j+\bar{o}(1))$. This is a
well-know result in the asymptotical theory of ODE but it is valid
for either large positive $\Im k$ or fixed $\Im k>0$ and small
$\|V\|_2$. It does not require any information on $Q_{11}$ and
$Q_{33}$. Notice that $b(T,k)=\exp(i\lambda_j kT)(1+O((\Im
k)^{-1}))$. Then, (\ref{iter}) yields invertibility of each $X_j$
for large $\Im k$. Also, since $\Delta_j$ is entire in $k$, the
formula
\[
b(T,k)=\Delta_j\Delta_{j-1}^{-1}
\]
allows to define $b$ for any $k$ as a meromorphic function.

For $k=iy,\quad y>>1$ we have the following asymptotical expansion

\[
\tilde{u}=f+Df+D^2f+\ldots
\]

\[
Df=\left[
\begin{array}{c}
\displaystyle -\int\limits_t^T U_1(\tau,t,iy)e^{-y\lambda_j(\tau-t)}Q_{12}(\tau) d\tau\\
0 \\
\displaystyle \int\limits_0^t
U_2(\tau,t,iy)e^{y\lambda_j(t-\tau)}Q_{32}(\tau)d\tau
\end{array}
\right]
\]
and
\[
\left(D^2f\right)_2(T)= \int\limits_0^T
\left[Q_{21}(\tau)(Df)_1(\tau)+Q_{23}(\tau)(Df)_3(\tau)\right]d\tau
\]
Denote $\Psi_2(\tau,t,y)=U_2(\tau,t,iy)e^{y\lambda_j(t-\tau)}$,
$\Psi_1(\tau,t,y)=U_1(\tau,t,iy)e^{y\lambda_j(t-\tau)}$.
Substituting the Duhamel expansions for $\Psi_{1(2)}$ into the
formula above, one gets
\begin{eqnarray*}
\left(D^2f\right)_2(T)=y^{-1}\left(\int\limits_0^T Q_{23}(t)
(\Lambda_c-\lambda_j)^{-1}Q_{32}(t)dt\right.
\\
\left. -\int\limits_0^T Q_{21}(t)
(\lambda_j-\Lambda_a)^{-1}Q_{12}(t)dt\right)+\bar{o}(y^{-1})
\end{eqnarray*}
Since $Q_{lj}=iV_{lj}$ and $|(D^3f)_2(T,iy)|=\bar{o}(y^{-1})$, we have
\[
b(T,iy)=\exp(-yT\lambda_j)\left(1+y^{-1}\left[\sum\limits_{k=1}^{j-1}
(\lambda_j-\lambda_k)^{-1}\int\limits_0^T|V_{kj}(\tau)|^2d\tau\right.\right.
\]
\[
\left.\left. -\sum\limits_{k=j+1}^{N}
(\lambda_k-\lambda_j)^{-1}\int\limits_0^T|V_{kj}(\tau)|^2d\tau
\right]+\bar{o}(y^{-1})\right)
\]
From (\ref{iter}), we have
\[
\ln \Delta_j(T,iy)=-(\lambda_1+\ldots+\lambda_j)yT
\]
\begin{equation}
-y^{-1}\sum\limits_{k=1}^{j}\sum\limits_{l=j+1}^N
|\lambda_l-\lambda_k|^{-1} \int\limits_0^T
|V_{kl}(\tau)|^2d\tau+\bar{o}(y^{-1})\label{magic}
\end{equation}
The similar calculation can be done in the general case when  $\Im
k\to+\infty$.

We are ready to prove the following
\begin{theorem}
If
\[
I(V)=\sum\limits_{k=1}^{j}\sum\limits_{l=j+1}^N
|\lambda_l-\lambda_k|^{-1} \int\limits_0^\infty
|V_{kl}(\tau)|^2d\tau<\infty
\]
then $g(t,k)=\Delta_j(t,k)\exp(-ikt(\lambda_1+\ldots+\lambda_j))$
converges in $\mathbb{C}^+$ to a function $g(\infty,k)$ which is in
the unit ball in $H^\infty(\mathbb{C}^+)$. Moreover,
\[
g(t,k)-g(\infty,k)\in H^2(\mathbb{C}^+), \quad \|g(t,k)-g(\infty,k)\|_2\to\infty
\]
and
\begin{equation}
0\geq \int\limits_{\mathbb{R}}\ln |g(\infty,k)|dk\geq -\pi I(V),
\quad \int\limits_{\mathbb{R}}|1-g(\infty,k)|^pdk\leq C(p)I(V),\quad
1<p<\infty \label{bns}
\end{equation}\label{main1}
\end{theorem}
\begin{proof}
Consider $g(t,k)$ for any $t>0$. It is entire in $k$ and
$|g(t,k)|\leq 1$ for real $k$ since $X_j$ is a contraction for $\Im
k\geq 0$. Moreover, we know its asymptotics for large $\Im k$ which
implies that $g(t,k)$ is in the unit ball in
$H^\infty(\mathbb{C}^+)$ and
\[
0\geq \int\limits_{\mathbb{R}}\ln |g(t,k)|dk\geq -\pi I(V)
\]
Arguing like in the proof of lemma \ref{lemka}, we write
$g(t,k)=1+h(t,k)$. Then $\Re h(t,k)\leq 0$ and

\[
\int\limits_{\mathbb{R}} \Re h(t,k)dk =-\pi I(V)
\]
which implies
\[
\|1-g(t,k)\|_{L^{1,w}(\mathbb{R})}\lesssim I(V), \quad
\|1-g(t,k)\|_{L^{p}(\mathbb{R})}\leq C(p)\left[I(V)\right]^{1/p},
\quad 1<p<\infty
\]\bigskip

Let $0\leq s_1(k,t)\leq s_2(k,t)\leq \ldots \leq s_j(k,t)\leq 1$ be
singular numbers of $X_j(t,k)$. Then,
\[
0\geq \int\limits_{\mathbb{R}} \Bigl(\ln s_1^2(k,t)+\ldots +\ln
s_j^2(k,t)\Bigr)dk\geq -2\pi I(V)
\]
and therefore
\begin{equation}
\int\limits_\mathbb{R} {\rm tr} (I_{j\times j}-|X_j(t,k)|^2)dk\leq
2\pi I(V) \label{h2n}
\end{equation}
Write $X(t,k)$ in the block form
\[
X(t,k)= \left[
\begin{array}{cc}
X_j & Y_j\\
Z_j & W_j
\end{array}
\right]
\]
Since $X$ is unitary, (\ref{h2n}) can be rewritten
\begin{equation}
\int\limits_{\mathbb{R}} {\rm tr} |Y_j(t,k)|^2dk\leq 2\pi I(V), \,
\int\limits_{\mathbb{R}} {\rm tr} |Z_j(t,k)|^2dk\leq 2\pi I(V)
\label{h2}
\end{equation}\bigskip

Now, that all necessary uniform bounds are obtained, we can prove
the convergence result. For any $t_1<t_2$, the semigroup property in
the block form yields the identity
\[
X_j(t_2,k)=X_j(t_1,t_2,k)X_j(t_1,k)+Y_j(t_1,t_2,k)Z_j(t_1,k)
\]
As $t_{1(2)}\to\infty$, we have
\[
\int_\mathbb{R} {\rm tr}|Y_j(t_1,t_2,k)|^2 dk\to
0
\]
 So,
\[
\int\limits_\mathbb{R}\left|\det X_j(t_2,k)-\det X_j(t_1,t_2,k)
\cdot \det X_j(t_1,k)\right|^2dk\to 0, \quad t_{1(2)}\to\infty
\]
On the other hand,
\[
\int\limits_{\mathbb{R}} \left|\det
X_j(t_1,t_2,k)-\exp(ik(t_2-t_1)(\lambda_1+\ldots+\lambda_j))\right|^2dk\to
0
\]
as $t_{1(2)}\to \infty$. So, $1-g(t,k)$ is Cauchy in
$H^2(\mathbb{C}^+)$ and we have convergence $g(t,k)\to g(\infty,k)$
uniformly over the compacts in $\mathbb{C}^+$. Since each $g(t,k)$
is analytic contraction, the limit $g(\infty,k)$ as an analytic
contraction as well. The bounds (\ref{bns}) can be obtained through
the argument identical to the one used to handle $g(t,k)$.
\end{proof}

{\bf Remark 6.} Notice that the theorem was proved under the
assumption that all eigenvalues of $\Lambda$ are non-degenerate.
That was used in the proof of the asymptotics for $b$. In the
meantime, due to cancelation in (\ref{magic}), the statement of
theorem \ref{main1} as well as (\ref{h2}) holds under the assumption
that $\lambda_j<\lambda_{j+1}$ and the other eigenvalues can
degenerate.\bigskip

The estimates for determinants and (\ref{h2}) can be obtained for
$W_{j-1}$ as well and that yields the following important result.

\begin{theorem}
If \textit{\[ I'(V)=\sum\limits_{k=1}^{j}\sum\limits_{l=j}^N
|\lambda_l-\lambda_k|^{-1} \int\limits_0^\infty
|V_{kl}(\tau)|^2d\tau<\infty, \quad ({\it remember\,\, that\,\,
}V_{jj}(t)=0 )
\]}
then
\[
\int\limits_\mathbb{R} |x_{jj}(t,k)-\exp(i\lambda_jtk)|^2dk\lesssim
I'(V)
\]
Moreover, there is $\tilde{x}_{jj}(\infty,k)$ such that
\[
\int_\mathbb{R} \left|x_{jj}(t,k)\exp(-i\lambda_jtk)
-\tilde{x}_{jj}(\infty,k)\right|^2dk\to 0
\]
and
\[
\int\limits_\mathbb{R} |\tilde{x}_{jj}(\infty,k)-1|^2dk\lesssim I'(V)
\]\label{main2}
\end{theorem}
\begin{proof}
The estimate (\ref{h2}), applied to $X_j$ and $W_{j-1}$, yields
\[
\int\limits_\mathbb{R} \sum\limits_{l\neq j}
\left(|x_{lj}(t,k)|^2+|x_{jl}(t,k)|^2\right)dk\lesssim I'(V)
\]
Expanding in the last raw, we have
\begin{equation}
\Delta_j=x_{jj}\Delta_{j-1}+r,\quad
r=x_{j1}A_{j1}+\ldots+x_{j,j-1}A_{j,j-1} \label{repr}
\end{equation}
where $\{A_{lm}\}$ are cofactors of $X_j$.
\begin{lemma}
If $Z$ is $j\times j$ contraction then the adjoint $C={\it adj} Z$
is contraction as well. In particular,
\begin{equation}
|A_{11}|^2+\ldots+|A_{1j}|^2\leq 1 \label{str}
\end{equation}
where $A_{ik}$ is $(i,k)$--cofactor of $Z$.\label{contr}
\end{lemma}
\begin{proof}
Take any $\alpha=(\alpha_1,\ldots \alpha_j)$ with $\|\alpha\|_2=1$.
Replace the first raw of $Z$ by $\alpha$ and denote the resulting
matrix by $Z_\alpha$. By Laplace theorem,
\[
\det Z_\alpha=\alpha_1A_{11}+\ldots +\alpha_jA_{1j}
\]
On the other hand,  Hadamard's estimate gives
\[
|\det Z_\alpha|\leq h_1\ldots h_j\leq 1
\]
where $h_l$ is the $\ell^2$--length of the $l$--th raw of
$Z_\alpha$. Since $\alpha$ is arbitrary, we get (\ref{str}). This
implies $\|C e_1\|\leq 1$. Take any unitary $U$. We have
$UCU^{-1}={\it adj}(UZU^{-1})$. Therefore,
\[
\|CU^{-1}e_1\|\leq 1
\]
Since $U$ is arbitrary, $\|Cx\|\leq \|x\|$ for any $x$.
\end{proof}
By lemma, we have
\[
\int\limits_\mathbb{R} |r(k)|^2dk\lesssim I'(V)
\]
Since
\[
\int
\left|\Delta_j-\exp(ikt(\lambda_1+\ldots+\lambda_j))\right|^2dk\lesssim
I'(V)
\]
and
\[
\int
\left|\Delta_{j-1}-\exp(ikt(\lambda_1+\ldots+\lambda_{j-1}))\right|^2dk\lesssim
I'(V)
\]
the formula (\ref{repr}) yields
\[
\int\limits_\mathbb{R}
\left|x_{jj}(t,k)-\exp(ikt\lambda_j)\right|^2dk\lesssim I'(V)
\]
Using the semigroup property one can show that
$x_{jj}(t,k)\exp(-ikt\lambda_j)-1$ is Cauchy in $L^2(\mathbb{R})$
which implies existence of the limit.
\end{proof}

\begin{corollary} Let
\[
\sum\limits_{k\neq l} \int\limits_0^\infty
|\lambda_k-\lambda_l|^{-1}|V_{kl}(t)|^2dt<\infty
\]
Consider $\tilde{X}(t,k)=\exp(-ikt\Lambda)X(t,k)$. Then
$\tilde{X}(t,k)\to \tilde{X}(\infty,k)$  in the strong sense, i.e.,
\[
\int\limits_\mathbb{R}
\|\tilde{X}(t,k)f-\tilde{X}(\infty,k)f\|^2dk\to 0
\]
for any $f$.
\end{corollary}
\begin{proof}
The proof is a standard application of the semigroup property and
the previous results.
\end{proof}

We are going to consider now a somewhat special case when the
frequencies degenerate in different ways. The first situation is a
model for Schr\"odinger evolution on $1$--d torus.

Assume that $\lambda_{j-1}<\lambda_j=\lambda_{j+1}<\lambda_{j+2}$
for some $j: 1<j<N$. We will try to understand how the
$P_{\{j,j+1\}}X(t,k)P_{\{j,j+1\}}$ part of $X(t,k)$ behaves for large $t$.
Consider the following evolutions
\[
\Psi'(\tau,t,k)=i\left(k\lambda_jI_{2\times 2}+\left[
\begin{array}{cc}
0 & V_{j,j+1}(t)\\
V_{j+1,j}(t) & 0
\end{array} \right]\right)\Psi(\tau,t,k),\quad
\Psi(\tau,\tau,k)=I_{2\times 2}
\]
\[
W'(\tau,t)=i\left[
\begin{array}{cc}
0 & V_{j,j+1}(t)\\
V_{j+1,j}(t) & 0
\end{array} \right]W(\tau,t),\quad
W(\tau,\tau)=I_{2\times 2}
\]
Obviously, $W$ is $k$--independent and is unitary since
$V_{j,j+1}=\bar{V}_{j+1,j}$. For $\Psi$, we have
$\Psi=\exp(ik\lambda_jt)W$.

Notice that for real or purely imaginary $V_{j,j+1}$ the matrix that
diagonalizes the perturbation is $t$--independent and we can assume
that $V_{j,j+1}(t)=0$ without loss of generality. We will consider
the general case.
 The proof of the following result repeats the argument
given above with minor changes which we will explain.

\begin{theorem}
Assume that $\lambda_{j-1}<\lambda_j=\lambda_{j+1}<\lambda_{j+2}$
and
\[
I''(V)=\sum\limits_{k=1}^{j-1}\sum\limits_{l=j}^{j+1}
|\lambda_l-\lambda_k|^{-1} \int\limits_0^\infty
|V_{kl}(\tau)|^2d\tau+\sum\limits_{k=1}^{j+1}\sum\limits_{l=j+2}^{N}
|\lambda_l-\lambda_k|^{-1} \int\limits_0^\infty
|V_{kl}(\tau)|^2d\tau<\infty,
\]
Consider $Y(t,k)=P_{\{j,j+1\}}X(t,k)P_{\{j,j+1\}}$. Then,
\begin{equation}
\int\limits_\mathbb{R} \|Y(t,k)-\Psi(0,t,k)\|^2dk\lesssim I''(V)
\label{est5}
\end{equation}
Moreover, $\|\Psi^{-1}(0,t,k)Y(t,k)- \tilde{Y}(\infty,k)\|\to 0$ in
$L^2(\mathbb{R})$ and
\[
\int\limits_\mathbb{R} \|\tilde{Y}(\infty,k)-I\|^2dk\lesssim I''(V)
\]\label{main3}
\end{theorem}
\begin{proof}
We repeat the proofs of theorems \ref{main1} and \ref{main2} with
the following modifications. Instead of a single vector $u$
satisfying $b(0,k)=1$, we consider its $N\times 2$ matrix version.
Let us denote the matrix containing first $j-1$ rows of $u$ by $a$,
$b$ is formed by $j, j+1$ rows and is therefore $2\times 2$ matrix,
and $c$ is built of $j+2, \ldots,N$ rows of $u$. Then, the boundary
conditions would be
\[
a(T,k)=0,\quad b(0,k)=I_{2\times 2},\quad c(0,k)=0
\]
The analogs of (\ref{ur}) and (\ref{vgw}) are
\[
X_{j+1}(T,k) \left[
\begin{array}{c}
a(0,k)\\ I_{2\times 2} \end{array}\right]= \left[
\begin{array}{c}
0\\ b(T,k) \end{array}\right]
\]
Multiplying from the left with the adjoint of $X_{j+1}(T,k)$ and
taking the $2\times 2$ blocks in the ``southeastern corner", we have
\begin{eqnarray}
\Delta_{j+1}(T,k) \cdot I_{2\times 2}= A(T,k)\cdot
b(T,k),\hspace{3cm}\nonumber \\\hspace{3cm} \quad\quad\quad\quad
A(T,k)=\left[
\begin{array}{cc}
A_{jj}(T,k) & A_{j+1,j}(T,k)\\
A_{j,j+1}(T,k) & A_{j+1,j+1}(T,k)
\end{array}
 \right]\label{mver}
\end{eqnarray}
By Remark 6, we already know that
\[
\int\limits_\mathbb{R}
|\Delta_{j+1}(T,k)-\exp(ikT(\lambda_1+\ldots+\lambda_{j-1}+2\lambda_j))|^2dk\lesssim
I''(V)
\]

\[
\int\limits_\mathbb{R}
|\Delta_{j-1}(T,k)-\exp(ikT(\lambda_1+\ldots+\lambda_{j-1}))|^2dk\lesssim
I''(V)
\]
and
\begin{eqnarray}
\int\limits_\mathbb{R} \sum\limits_{l\neq j,j+1}
\Bigl[|x_{l,j}(T,k)|^2+|x_{l,j+1}(T,k)|^2\nonumber
\hspace{3cm}\\\hspace{3cm} +|x_{j,l}(T,k)|^2+|x_{j+1,l}(T,k)|^2
\Bigr] dk\lesssim I''(V) \label{borders}
\end{eqnarray}

Combining these estimates and using the Laplace theorem for determinants, we
have

\begin{equation}
\int\limits_\mathbb{R} |\det Y(T,k)-\exp(2ikT\lambda_j)|^2dk\lesssim
I''(V) \label{dete}
\end{equation}

 The analog of (\ref{inteq}) is
\begin{equation}
\left\{
\begin{array}{l}
\displaystyle a(t)=-\int_t^T U_1(\tau,t,k)\left[
Q_{12}(\tau)b(\tau)+Q_{13}(\tau)c(\tau)\right]d\tau
\\
\displaystyle b(t)=\Psi(0,t,k)+\int_0^t \Psi(\tau,t,k)\left[
Q_{21}(\tau)a(\tau)+Q_{23}(\tau)c(\tau)\right]d\tau
\\
\displaystyle c(t)=\int_0^t U_2(\tau,t,k)\left[
Q_{31}(\tau)a(\tau)+Q_{32}(\tau)b(\tau)\right]d\tau
\end{array}\right.\label{inteq1}
\end{equation}
The similar perturbation argument gives
\[
b(T,iy)=\Psi(0,T,iy)\left(I_{2\times
2}+y^{-1}\Gamma+\bar{o}(y^{-1})\right), \quad y\to +\infty
\]
where
\begin{eqnarray*}
\Gamma=\int_0^T
W(t,0)Q_{23}(t)(\Lambda_c-\lambda_j)^{-1}Q_{32}(t)W(0,t)dt\hspace{2cm}
\\
-\int_0^T
W(t,0)Q_{21}(t)(\lambda_j-\Lambda_a)^{-1}Q_{12}(t)W(0,t)dt
\end{eqnarray*}
Consider
\begin{equation}
C(T,k)=A(T,k)\Psi(0,T,k)\exp(-ikT(\lambda_1+\ldots
+\lambda_{j-1}+2\lambda_j)) \label{oprc}
\end{equation}
Since we know asymptotical expansion for $\Delta_{j+1}$ as $\Im k\to +\infty$,
(\ref{mver}) yields
\[
C(T,iy)=I_{2\times 2}-y^{-1}(\Gamma+I_{j+1}(V)\cdot I_{2\times
2})+\bar{o}(y^{-1})
\]
Notice that \[0\leq \Gamma+I_{j+1}(V)\cdot I_{2\times
2} \lesssim  I''(V)\]

By lemma \ref{contr}, $A(T,k)$ is contraction for $\Im k\geq 0$ and so is $C(T,k)$
 for real $k$. $C(T,k)$ is also entire in $k$ and we know its
asymptotics for large $\Im k$ which implies that $C(T,k)$ is contraction for $k\in \mathbb{C}^+$. Write $C(T,k)=I_{2\times 2}+H(T,k)$.
Then,
\begin{equation}
2\Re H(T,k)+|H(T,k)|^2\leq 0 \label{sos}
\end{equation}
Since $(\Re H(T,k)\xi,\xi)$ is harmonic for any $\xi\in \mathbb{C}^2$, comparison of
asymptotics for $y\to\infty$ gives
\[
-\int\limits_\mathbb{R} \Re H(T,k)dk\lesssim  I''(V)
\]
and then (\ref{sos}) yields
\begin{equation}
\int\limits_\mathbb{R}  \|H(T,k)\|^2dk\lesssim I''(V) \label{eh}
\end{equation}
Next, notice that
\[
A(T,k)=\Delta_{j-1}(T,k)\left[
\begin{array}{cc}
x_{j+1, j+1}(T,k) & -x_{j,j+1}(T,k)\\
-x_{j+1,j}(T,k) & x_{j,j}(T,k)
\end{array}
 \right]+r(T,k)
\]
and
\[
\int\limits_\mathbb{R} \|r(T,k)\|^2dk\lesssim I''(V)
\]
due to (\ref{borders}). On the other hand, we know the asymptotics
of $\Delta_{j-1}(T,k)$ which together with  (\ref{oprc}) and (\ref{eh}) give
\begin{equation}
\int\limits_\mathbb{R} \|\exp(2ikT\lambda_j)-({\rm
adj\,}Y(T,k))\cdot\Psi(0,T,k)\|^2dk\lesssim I''(V) \label{redka}
\end{equation}
since
\[
{\rm adj}\, Y(T,k)=\left[
\begin{array}{cc}
x_{j+1, j+1}(T,k) & -x_{j,j+1}(T,k)\\
-x_{j+1,j}(T,k) & x_{j,j}(T,k)
\end{array}
 \right]
\]
Denote the matrix under the norm in (\ref{redka}) by $\mu$. We have $\|Y\mu\|\leq
\|\mu\|$ since $Y$ is a contraction and therefore
\[
\int\limits_\mathbb{R} \|\exp(2ikT\lambda_j)Y(T,k)-\det Y(T,k)\cdot
\Psi(0,T,k)\|^2dk\lesssim I''(V)
\]
By (\ref{dete}), we have (\ref{est5}). The rest is standard.
\end{proof}
The method can be carried over to the case when the multiplicity of
frequencies are higher than $2$. Using these results we can obtain
the ``asymptotics'' of solution for any $\lambda_1\leq \ldots\leq
\lambda_N$ provided that $\|V\| \in L^2(\mathbb{R}^+)$. However, the
constants in our estimates will blow up when some $\delta_j=\lambda_{j+1}-\lambda_j\sim 0$.\bigskip

Assume we are in the situation when
$\lambda_1=\lambda_2=\ldots=\lambda_m<\lambda_{m+1}$. The simple
matrix version of lemma \ref{lemochka} gives
\begin{proposition}
If $\lambda_1=\lambda_2=\ldots=\lambda_m<\lambda_{m+1}$ and
\[\tilde{v}(t)=\|P_{\{1,\ldots, m\}}V(t)P_{\{m+1,\ldots\}}\|\in
L^2(\mathbb{R}^+)\] then

\[
\int\limits_\mathbb{R} \|P_{\{m+1,\ldots\}}X(T,k)e_n\|^2dk \lesssim
\|\tilde{v}\|^2\lambda_{m+1}^{-1}, \quad n=1,\ldots, m
\]
\end{proposition}
If $\lambda_{m+1}>>1$, this means rather strong localization of the
solution. Therefore, we pose the following

{\bf Open problem.} Is it possible to improve the estimate on
\[
\int\limits_\mathbb{R} \|P_{\{m,\ldots\}} X(T,k)e_1\|^2dk
\]
for large $m$ assuming only  $\lambda_1<\lambda_2<\ldots$ and some
off-diagonal decay  for $V$? The conjecture might be that
\begin{equation}
\sup_{T>0}\int\limits_\mathbb{R} \|\sqrt\Lambda
X(T,k)e_1\|^2dk<\infty \label{gen}
\end{equation}
for suitable $L^2$ condition on $V$. That could lead to better
understanding of  Schr\"odinger evolution on the circle. \bigskip

 The calculations below will be extensively used in later sections to handle special evolutions equations. We will empasize the dependence of $I'(V)$ on $j$ by writing $I'_j(V)$.
\begin{itemize}
\item[(a)]
Let $V_{ij}(t)=q_{i-j}(t)$ where $q_j(t)=\overline{q_{-j}(t)}$ and
\[
\|q\|_2^2=\int\limits_0^\infty \sum\limits_{l=1}^N
|q_l(t)|^2dt<\infty
\]
Assume also that $|\lambda_l-\lambda_m| \gtrsim |l-m|$. Then for any
 $j$ we have (remember that $q_0(t)=0$)
\[
I'_j(V)\lesssim \sum\limits_{k\leq 0} \sum\limits_{l\geq 0}
\frac{1}{|k-l|} \int\limits_0^\infty |q_{k-l}(t)|^2dt\lesssim
\|q\|_2^2
\]

\item[(b)] Take the same $V$ but assume that $\lambda_j\sim j^m$.
Then,
\begin{equation}
I'_j(V)\lesssim \sum\limits_{k=1}^j \sum\limits_{l=j}^\infty
\frac{1}{l^m-k^m} \int\limits_0^\infty |q_{k-l}(t)|^2dt\lesssim
j^{-(m-1)} \|q\|_2^2 \label{calc1}
\end{equation}

\end{itemize}

In the second case, the condition on $V$ can be relaxed. If $\lambda_{2j}=\lambda_{2j+1}\sim j^m$, the estimate for $I''_j(V)$ is similar.

\section{The case of Hilbert spaces and applications to Schr\"{o}dinger evolution on the circle}

 In this section, some results from the previous section are
 generalized to the case $N=\infty$.
 We will also give various applications to the
Schr\"odinger evolution on the circle.

Consider the selfadjoint operator $\Lambda$ on
$\cal{H}=\ell^2(\mathbb{Z})$ with discrete spectrum $\{\lambda_n\}$
where $\lambda_n$ is nondecreasing sequence. Let $Q(t)$ be
operator-valued function with norm $\|Q(t)\|$ bounded for a.e. $t$
and
\begin{equation}\|Q(t)\|\in L^1_{\rm
loc}(\mathbb{R}^+)\label{regw}\end{equation}

The weak solution to
\begin{equation}
u_t(t,k)=ik\Lambda u(t,k)+Q(t)u(t,k), \quad u(0,k)=\psi \label{weak}
\end{equation}
is  the solution to
\[
u(t,k)=X_0(0,t,k)\psi+\int_0^t X_0(\tau,t,k)Q(\tau)u(\tau,k)d\tau
\]
which follows from the Duhamel formula. Here $X_0(\tau,t,k)$ denotes
solution to the unperturbed evolution. We will write
$u(t,k)=X(0,t,k)\psi$.

There are some general results that prove the weak solution is in
fact a ``strong solution" provided that the initial value and
potential $Q$ are ``regular enough". We will study the behavior of
weak solution. The condition (\ref{regw}) is sufficient for the
iterations of (\ref{weak}) to converge in the space $L^\infty
([0,T],\cal{H})$ for any $T>0$ with the obvious estimate
\[
\|u(t,k)\|\leq \exp\left(\int_0^t \|Q(\tau)\|d\tau\right) \|\psi\|
\]
The following stability result will allow us to use the standard
approximation technique. Let $\Pi_n=P_{\{-n,\ldots, n\}}$, a
projection in $\ell^2(\mathbb{Z})$.
\begin{lemma}{\rm\bf  (Approximation lemma).}
If $\|Q(t)\|\in L^1(0,T)$ and   $Q_n(t)=\Pi_n Q(t)\Pi_n$, then
$\|u_n(T,k)-u(T,k)\|\to 0$ as $n\to \infty$.
\end{lemma}
\begin{proof}
Each term in the corresponding series is a multilinear operator
\begin{eqnarray*}
G_m(Q,\ldots, Q)=\int_0^T\int_0^{\tau_1}\ldots\int_0^{\tau_{m-1}}
X_0(\tau_1,T,k)Q(\tau_1)X_0(\tau_2,\tau_1,k)\ldots\\
Q(\tau_m)X_0(0,\tau_m,k)\psi d\tau_1\ldots d\tau_m
\end{eqnarray*}
or
\begin{eqnarray*}
{G}_m(Q_n,\ldots,
Q_n)=\int_0^T\int_0^{\tau_1}\ldots\int_0^{\tau_{m-1}}
X_0(\tau_1,T,k)Q_n(\tau_1)X_0(\tau_2,\tau_1,k)\ldots\\
Q_n(\tau_m)X_0(0,\tau_m,k)\psi d\tau_1\ldots d\tau_m
\end{eqnarray*}
Write
\[
{G}_m(Q_n,\ldots, Q_n)={G}_m(Q_n,\ldots, Q_n, Q)+{G}_m(Q_n,\ldots,
Q_n, Q_n-Q)
\]
and use linearity for the second term, etc. Then,
\[
{G}_m(Q_n,\ldots, Q_n)-{G}_m(Q,\ldots, Q)
\]
can be written as a sum of $m$ terms and each of them converges to
zero. Indeed,
\begin{eqnarray*}
\int_0^T\ldots \int_0^{\tau_{m-1}} \|Q_n(\tau_1)\|\cdot \ldots \cdot
\|Q_n(\tau_{j-1}\|\cdot
\|\Bigl(Q(\tau_j)-Q_n(\tau_j)\Bigr)X_0(\tau_{j+1},\tau_j,k)Q(\tau_{j+1})\\
\ldots Q(\tau_m)X_0(0,\tau_m,k)\psi\|d\tau_1\ldots d\tau_m\to 0
\end{eqnarray*}
by dominated convergence theorem. Therefore
\[
\sum_{m\geq 1} G_m(Q_n,\ldots, Q_n)\to \sum_{m\geq 1} G_m(Q,\ldots,
Q), \quad n\to \infty
\]
since we also have a bound
\[
|G_m(V,\ldots,V)|\leq \frac{1}{m!}\left(\int_0^T
\|V(t)\|dt\right)^m\|\psi\|
\]
which takes care of the tails in the series.
\end{proof}

It is now easy to prove
\begin{lemma}
If $Q(t)=iV(t)$, where $V(t)$ is selfadjoint and $\|V(t)\|\in
L^1_{\rm loc}(\mathbb{R}^+)$, then $X(\tau,t,k)$ is unitary and it
satisfies the semigroup property
\[
X(t_1,t_2,k)\cdot X(t_0,t_1,k)=X(t_0,t_2,k)
\]
\end{lemma}
\begin{proof}
The semigroup property and preservation of the norm follow from the
Approximation lemma and the corresponding results for finite systems
of ODE's. We also have
\[
X(0,t,k)\cdot X(t,0,k)=X(t,0,k)\cdot X(0,t,k)=I
\]
which implies that $X$ is unitary.
\end{proof}

Next, we  prove an  analog of theorem \ref{main2}. Denote the matrix
elements of $V(t)$ and $X(t,k)$ by $V_{mn}(t)$ and $x_{mn}(t,k)$,
respectively. For simplicity, we again make an assumption that
$V_{nn}(t)=0$ for any $n$.
\begin{theorem}
Assume that $\|V\|\in L^1_{\rm loc}(\mathbb{R}^+)$, $\lambda_{-1}<\lambda_0=0<\lambda_1$, and
\[
I'(V)=\sum\limits_{k\leq 0}\sum\limits_{l\geq 0}
|\lambda_l-\lambda_k|^{-1} \int\limits_0^\infty
|V_{kl}(\tau)|^2d\tau<\infty, \quad (V_{00}(t)=0)
\]
Then,
\[
\int\limits_\mathbb{R} |x_{00}(t,k)-x_{00}(\infty,k)|^2dk\to 0
\]
\[
\int\limits_\mathbb{R} |x_{00}(\infty,k)-1|^2dk\lesssim I'(V)
\]\label{apat}
\end{theorem}

\begin{proof}
For truncated potentials $V^{(n)}=\Pi_n V\Pi_n$, the theorem
\ref{main2} is applicable and the resulting estimates are uniform in
$n$. Since for each fixed $k$ we have convergence
\[
x^{(n)}_{jl}(t,k)\to x_{jl}(t,k), \quad n\to\infty
\]
and uniform estimates
\[
\int\limits_\mathbb{R} |x_{00}^{(n)}(t,k)-1|^2dk\lesssim I'(V)
\]
\[
\int\limits_\mathbb{R} \sum\limits_{l\neq 0}
(|x_{0l}^{(n)}(t,k)|^2+|x_{l0}^{(n)}(t,k)|^2)dk\lesssim I'(V)
\]
we can go to the limit as $n\to\infty$ to get
\[
\int\limits_\mathbb{R} |x_{00}(t,k)-1|^2dk\lesssim I'(V), \quad
\int\limits_\mathbb{R} \sum\limits_{l\neq 0}
(|x_{0l}(t,k)|^2+|x_{l0}(t,k)|^2)dk\lesssim I'(V)
\]
The rest is the standard application of semigroup property and unitarity of $X$.
\end{proof}
Most results from the previous section can be adjusted similarly
including the case when the frequencies have multiplicity. In
particular, we can consider Schr\"odinger evolution on, say,
one-dimensional circle
\[
u_t=-iku_{\theta\theta}+iV(t,\theta)u, \quad u(0,k)=\psi(\theta)\in L^2(\mathbb{T}).
\]
We are interested in the weak solution and assume that $V$ is
real-valued and $\|V(t,\theta)\|_\infty\in L^1_{\rm
loc}(\mathbb{R}^+)$. In this case, on the Fourier side, equation
takes  form
\[
\hat{u}_t=ik\Lambda \hat{u}+i\hat{V}\ast \hat{u}, \quad \hat{u}(0,k)=\hat\psi\in \ell^2
\]
where $\Lambda$ is diagonal $\lambda_0=0, \lambda_n=n^2$ and all
eigenvalues but the principal one (i.e., $\lambda_0=0$) have
multiplicity $2$. Write $\Lambda$ in the basis $\{ 1,e^{i\theta},
e^{-i\theta}, \ldots, e^{in\theta}, e^{-in\theta}, \ldots\}$
\[
\Lambda=\left[
\begin{array}{cccccc}
0 & 0 & 0 & 0 & 0 &\ldots \\
0 & 1 & 0 & 0  & 0 &\ldots \\
0 & 0 & 1 & 0   & 0 &\ldots \\
0 & 0 & 0 & 4 & 0 & \ldots \\
0 &  0 &  0 &  0 &   4 &\ldots\\
\ldots & \ldots &  \ldots & \ldots & \ldots & \ldots
\end{array}
\right]
\]

 We also always assume without loss of generality
that
\[
\int\limits_{\mathbb{T}} V(t,\theta)d\theta=0, \quad t>0
\]
Consider the following $k$--independent evolution
\[
\Psi_n'(t)=i\left[
\begin{array}{cc}
0 & \hat{V}(2n,t)\\
\overline{\hat{V}}(2n,t) & 0
\end{array}
\right]\Psi_n(t), \quad \Psi_n(0)=I_{2\times 2}, \quad n\geq 1
\]
where
\[
\hat{V}(2n,t)=\int\limits_{\mathbb{T}} V(t,\theta)e^{2in\theta}d\theta
\]
Notice that if $\hat{V}(2n,t)=0$ for all $t$, then $\Psi_n(t)$ are
trivial. Also, if $\hat{V}(2n,t)$ is real or purely imaginary, we
can write the explicit formula for $\Psi_n(t)$. That can be
satisfied, e.g., if $V$ is even or odd.

 Take
\[
W(0,t,k)=X_0(0,t,k)\cdot
\left[
\begin{array}{cccc}
1 & 0 & 0 & \ldots \\
0 & \Psi_1(t) & 0 &\ldots\\
0 & 0 & \Psi_2(t) & \ldots
\end{array}
\right]
\]
where $X_0(0,t,k)$ is the free Schr\"odinger evolution.

\begin{theorem}
Assume that
\begin{equation}
\|V(t,\theta)\|_\infty\in L^1_{\rm loc}(\mathbb{R}^+) \label{nons}
\end{equation}
and $V(t,\theta)\in L^2(\mathbb{R}^+\times \mathbb{T})$. Then for
any $\psi\in L^2(\mathbb{T})$ the weak solution $u(t,k)$ satisfies
\[
\int\limits_I \|W^{-1}(0,t,k)\hat{u}(t,k)-\widehat{H\psi}\|^2dk\to
0, \quad {\rm as} \quad t\to\infty
\]
for any $I\subset \mathbb{R}, |I|<\infty$. The operator $H$ is
defined as bounded operator from $L^2(\mathbb{T})$ to the space of
functions $h(\theta,k)$ satisfying $
\|h(\theta,k)\|^2_{L^2(\mathbb{T})}\in L^1_{\rm loc}(\mathbb{R})$.
\end{theorem}
\begin{proof}
 Denote the
matrix elements of $\tilde{X}(t,k)=W^{-1}(0,t,k)X(t,k)$ by
$\tilde{x}_{mn}(t,k)$. Fix $n$ and take, say,  $\psi(\theta)=e^{
in\theta}$. Let $\hat{u}$ be the corresponding solution, i.e. the
$\alpha(n)$--th column of $\tilde{X}$. The choice
$\psi(\theta)=e^{-in\theta}$ gives the $\alpha+1$--th column.
 From the Approximation lemma, theorem~\ref{apat} (adapted by the theorem \ref{main3}), and (\ref{calc1}) we know

\[
\sum\limits_{m\in \mathbb{Z}^+} |\tilde{x}_{m\alpha}(t,k)|^2=1,
\quad t>0, k\in \mathbb{R}
\]
\[
\int\limits_\mathbb{R} |1-\tilde{x}_{\alpha\alpha}(t,k)|^2dk\lesssim
\|V\|^2,\quad \int\limits_\mathbb{R}
|\tilde{x}_{\alpha\alpha}(t,k)-\tilde{x}_{\alpha\alpha}(\infty,k)|^2dk\to
0
\]
and
\[
\int\limits_\mathbb{R} \sum\limits_{m\neq \alpha}
|\tilde{x}_{m\alpha}(t,k)|^2dk\lesssim \|V\|^2, \quad
\int\limits_\mathbb{R}
|\tilde{x}_{m\alpha}(t,k)-\tilde{x}_{m\alpha}(\infty,k)|^2dk\to 0,
\quad m\neq \alpha
\] Moreover, (\ref{h2}) gives the following uniform estimates
\begin{equation}
\int\limits_\mathbb{R} \sum\limits_{|m|>l} |\tilde{x}_{m\alpha
}(t,k)|^2dk\lesssim C(n) l^{-1}\|V\|^2, \quad l>>n \label{sss}
\end{equation}
which implies
\begin{equation}
\int\limits_\mathbb{R} \sum\limits_{m\neq \alpha}
|m|^\gamma|\tilde{x}_{m\alpha}(t,k)|^2dk\lesssim C(n)\|V\|^2
\label{smooth}
\end{equation}
for any $\gamma<1$. Therefore,
\[
\int\limits_\mathbb{R} \sum_{m> 0}
|\tilde{x}_{m\alpha}(t,k)-\tilde{x}_{m\alpha}(\infty,k)|^2dk\to 0
\]
By linearity, we can prove existence of the limit for any
 trigonometric polynomial
$\psi=T(\theta)$. Denote the corresponding limit by
$[HT](\theta,k)$. That gives a linear operator $H$ defined on the
set dense in $L^2(\mathbb{T})$. We have
\[ \|\tilde{X}(t,k)\hat{T}\|=\|T\|
\]
for any $k$ and $t$ and so for any $I\subset \mathbb{R}$
\[
\int\limits_I\int\limits_\mathbb{T} |[HT](\theta,k)|^2d\theta
dk=|I|\cdot \|T\|^2
\]
Therefore, $H$ can be extended to a bounded operator on
$L^2(\mathbb{T})$ such that
\[
\int\limits_I\int\limits_\mathbb{T} |[H\psi](\theta,k)|^2d\theta
dk=|I|\cdot \|\psi\|^2
\]
If $\psi $ is fixed, the last identity implies that
$\|H\psi\|=\|\psi\|$ for a.e. $k$.\end{proof}

Notice that the condition (\ref{nons}) was used only to guarantee
the global existence of the weak solution for {\it any} $k$ and can
probably be dropped. The solution corresponding to the initial value
$\psi=1$ is special in a way that we always have
\[
\int_\mathbb{R} \ln |x_{11}(t,k)|dk\gtrsim -\|V\|_2^2
\]
and that means $|x_{11}(t,k)|>0$ for a.e. $k$ (compare with
(\ref{never})).

\bigskip

The following proposition is the direct corollary from
(\ref{smooth})
\begin{proposition}
Take $\psi=1$ and assume that $\|V(t,\theta)\|_\infty\in L^1_{\rm
loc}(\mathbb{R}^+)$ and $V(t,\theta)\in L^2([0,T]\times \mathbb{T})$
for any $T>0$. Then,
\[
\int\limits_\mathbb{R}
\|u(T,\theta,k)\|^2_{\dot{H}^\gamma(\mathbb{T})}dk\lesssim
\int\limits_0^T \int\limits_\mathbb{T}|V(t,\theta)|^2d\theta dt  \]
for any $\gamma<1/2$.
\end{proposition}

The analogous bound can be proved for any sufficiently smooth
function $\psi$. Assuming that $V$ is only bounded on the strip
$\mathbb{R}^+\times \mathbb{T}$ this estimate shows that
$k$--averaged  $H^\gamma$ norm is finite and grows not faster than
$\sqrt{t}$.\bigskip

{\bf Remark 7.} Now, assume that
\[
\int\limits_0^T \int\limits_\mathbb{T}|V(t,\theta)|^2d\theta dt  <C
\]
Consider the first $T$ columns in $X(0,T,k)$ and denote by $M$ the
set of those for which
\[
\int_{\mathbb{R}} \sum_{l=T+1}^\infty |x_{lm}(0,T,k)|^2dk>\sigma
T^{-2}
\]
Due to (\ref{h2}), we have $|M|<C\sigma^{-1} T$ and by taking
$\sigma$ large we have that at least a half of the first $T$ columns
are strongly localized for many $k$. In this argument, $M$ can
depend on $T$, in principle. \bigskip

 In the case of transport equation, our method allows us to
reproduce that the solutions are in $H^{1/2}$. Indeed, we have
\[
\int\limits_\mathbb{R} \sum\limits_{m\leq 0, n\geq 0, m\neq n}
|x_{mn}(t,k)|^2dk\lesssim \|V\|^2_2
\]
where $x_{mn}(t,k)$ are the matrix elements of the evolution
operator in the Fourier representation. In the meantime, we have
$x_{mn}(t,k)=\exp(ikmt)x_{n-m,0}(t,k)$ which yields
\[
\int\limits_\mathbb{R} \|X(t,k)\cdot
1\|^2_{\dot{H}^{1/2}(\mathbb{T})}dk=\int\limits_\mathbb{R}
\sum\limits_{n\neq 0} |n|\cdot|x_{n0}(t,k)|^2dk\lesssim \|V\|_2^2
\]

\section{Evolution with deteriorating gap condition: the short-range interactions.}

This section contains the main results of the paper. Unfortunately,
they handle only the short-range potentials and even in this case
are far from optimal.

Consider, e.g., the following model
\begin{equation}
u_t=-ikt^{-2}u_{\theta\theta}+iV(t,\theta)u, \quad t>1 \label{dt}
\end{equation}
$V$ is real and
\[
u(1,\theta)=1
\]
Similar evolution equation appears as the WKB correction in the
three-dimensional Schr\"odinger dynamics \cite{den2}. We assume that
$V$ is real-valued and satisfies
\[
\|V(t,\theta)\|_{L^\infty(\mathbb{T})}\lesssim t^{-\gamma},
\]
where $0\leq \gamma\leq 1$ is to be specified later.

 On the Fourier
side, the equation can be written as
\[
\hat{u}'=ikt^{-2}\Lambda\hat{u}+i\hat{V}\ast \hat{u},\quad
\hat{u}(1)=\delta_0
\]
and $\Lambda$ is diagonal with elements $n^2$, $n\geq 0$. The
multiplicity of each eigenvalue is two as long as $n>0$, the
principal eigenvalue is non-degenerate. Clearly, this case can not
be handled by the methods considered in the previous section since
the distance between eigenvalues decays like $t^{-2}$ which might
lead to significant growth of the Sobolev norms even for ``typical"
$k$. Instead, as results of the previous section suggest, we should
introduce the scaled Sobolev norms

\[
\|u\|_{s,t}=t^{-s}\|u\|_{\dot{H}^s(\mathbb{T})}
\]

{\bf Open problem.} Assume $V(t,\theta)\in L^2([0,\infty)\times
\mathbb{T})$. Is it true that for a.e. $k$ we have
\[
\|u\|_{1,t}\to 0
\]
as $t\to\infty$?\bigskip

This conjecture is supported, e.g., by calculations (\ref{evid})
made for transport equation or by the Remark 7. If true, it implies
\[
\sum\limits_{|n|> Ct} |\hat{u}(t,n)|^2\to 0
\]
for any $C$ and since $\|u\|_2=1$, we have
\[
\sum\limits_{|n|< Ct} |\hat{u}(t,n)|^2\to 1
\]
so the most of $L^2(\mathbb{T})$ norm is concentrated on, roughly,
$t$ first harmonics. We will call this phenomenon the concentration
of $L^2$ norm. It does not seem to be possible to obtain any
asymptotical result similar to the case when the gap condition does
not deteriorate and, perhaps, the ``scattering" for this model
should be defined in terms of the boundedness of scaled Sobolev
norms.\bigskip

  The simple substitution $\tau=t^{-1}$  reduces the problem
to equation
\[
\psi_\tau=ik\psi_{\theta\theta}-i\tau^{-2}V(\tau^{-1},\theta)\psi,
\quad \psi(1)=1,\quad  0<\tau<1
\]
and for $q(\tau,\theta)=-\tau^{-2}V(\tau,\theta)$ we have the
following bound as $\tau\to +0$

\[
\|q(\tau,\theta)\|_{L^\infty(\mathbb{T})}\lesssim \tau^{\gamma-2}
\]
Thus, (\ref{dt}) can be reduced to studying the standard problem on
the circle
\begin{equation}
\psi_t=ik\psi_{\theta\theta}+iq(t,\theta)\psi, \quad
\psi(t,\theta)=1 \label{tough}
\end{equation}
where the potential grows in the controlled way. We will study the
growth of the standard Sobolev norm. Assume for a second that we
could prove ({\bf which we can not!} but compare to (\ref{gen}))
\begin{equation}
\int\limits_{\mathbb{R}}
\|\psi_\theta(t,\theta,k)\|_{L^2(\mathbb{T})}^2dk\lesssim
\int\limits_0^t \|q(\tau,\theta)\|^2_{L^2(\mathbb{T})}d\tau
\label{con}
\end{equation}
Then, for the original problem that would mean
\begin{equation}
\int\limits_{\mathbb{R}}
\|u_\theta(T,\theta,k)\|^2_{L^2(\mathbb{T})}dk\lesssim
\int\limits_0^T t^2\|V(t,\theta)\|^2_{L^2(\mathbb{T})}dt
\end{equation}
and so
\[
\int\limits_\mathbb{R} \|u(T,\theta,k)\|_{1,T}^2dk
=T^{-2}\int\limits_0^T t^2\|V(t,\theta)\|^2_{L^2(\mathbb{T})}dt\to 0
\]
provided that $V\in L^2([0,\infty)\times \mathbb{T})$. Notice also
that by the standard time--scaling it would be sufficient to prove
(\ref{con}) only for $t=1$.\bigskip\bigskip

We will start with rather simple apriori estimates. Consider the
simplified version of (\ref{dt})
\begin{equation}
u_t=ikT^{-2}u_{\theta\theta}+iV(t,\theta)u, \quad
0<t<T,\quad|V(t,\theta)|\lesssim T^{-\gamma}, \quad u(0,\theta)=1
\label{dt1}
\end{equation}

We start with well-known estimate
\begin{lemma}
Assume that  $V(t,\theta)$ is real trigonometric polynomial of
degree smaller than $T^\alpha$ for any $t\in [0,T]$ and
$\quad|V(t,\theta)|\lesssim T^{-\gamma}$. Then, for any $k\in
\mathbb{R}$
\[
T^{-1}\|u(T)\|_{\dot{H}^1(\mathbb{T})}\lesssim T^{\alpha-\gamma}
\]
Therefore,
\[
\sum\limits_{|n|>CT} |u_n(T)|^2\lesssim  T^{\alpha-\gamma}
\]\label{prim1}
\end{lemma}
\begin{proof}
The proof is elementary. Differentiating (\ref{dt1}) in angle,
multiplying by $\bar{u}_\theta$ and integrating, we get
\[
\|u_\theta(t)\|_2^2\leq 2\int\limits_0^t
\|V_\theta(\tau)\|_{L^\infty(\mathbb{T})}\|u_\theta(\tau)\|_2d\tau
\]
If $\max\limits_{t\in [0,T]} \|u_\theta(t)\|_2=\|u_\theta(t_m)\|_2$,
then
\[
\|u_\theta(t_m)\|_2\leq 2\int\limits_0^T \|V_\theta(t)\|_\infty
dt\leq 2T^{1+\alpha-\gamma}
\]
by Bernstein.
\end{proof}
Clearly, we have concentration of $L^2$ norm for all $k$ as long as
$\alpha<\gamma$. This argument holds for transport equation as well
and can be easily modified to control the higher Sobolev norms. On
the other hand, for the transport equation, the $L^2$ norm can
really smear over first $T^{1-\gamma}$ harmonics as can be easily
seen from van der Corput lemma applied to (\ref{inst}).\bigskip

If one writes
\[
u(t,\theta)=\exp\left(i \int_0^t
V(\tau,\theta)d\tau\right)\psi(t,\theta)
\]
in the previous lemma, then the equation for $\psi$ reads
\[
\psi_t=ikT^{-2}\psi_{\theta\theta}+I,
\]
where

\[
I=-2kT^{-2}\psi_\theta\int_0^t V_\theta(\tau,\theta)d\tau+ik\psi
T^{-2} \left(i\int_0^t
V_{\theta\theta}(\tau,\theta)d\tau-\left(\int_0^tV_\theta(\tau,\theta)d\tau\right)^2\right)
\]
For $I$, we have
\[
\|I(t)\|_2\lesssim  T^{2(\alpha-\gamma)}+T^{2\alpha-\gamma-1}
\]
by the previous lemma.

Thus, if $1+2\alpha<2\gamma$, then $\|\psi(T,\theta)-1\|_2\lesssim
T^{1+2\alpha-2\gamma}\to 0$ by Duhamel formula which proves the
standard WKB asymptotics of solution for the range
$\alpha<\gamma-1/2$.\bigskip\bigskip

In the case just considered, the potential had an extra smoothness
in $\theta$. The other extreme case is when $V$ oscillates.
\begin{lemma}
Assume that  $V(t,\theta)$ is real trigonometric polynomial,
$|V(t,\theta)|\lesssim T^{-\gamma}$, and $\hat{V}(n,t)=0$ for
$|n|<T^\alpha$ and $|n|>CT$. Then,
\[
\int_\mathbb{R} |\hat{u}_0(T,k)-1|^2dk\lesssim T^{5-2\alpha-4\gamma}
\]
\begin{proof}
On the Fourier side, apply the Duhamel formula to
$\hat{u}(t,k)=\exp(ikT^{-2}\Lambda t)\psi(t,k)$ to get
\[
\psi(t,k)=\delta_0+i\int_0^t e^{-ikT^{-2}\Lambda
\tau}\hat{V}(\tau)e^{ikT^{-2}\Lambda \tau}\psi(\tau,k)d\tau
\]
Taking the scalar product with $\delta_0$ and integrating by parts
\[
\langle \psi(T,k), \delta_0\rangle=1+I,
\]
where \[ I= i\int_0^T\langle \psi'(t,k),\int_t^T e^{-ikT^{-2}\Lambda
\tau}\hat{V}(\tau)e^{ikT^{-2}\Lambda \tau}\delta_0d\tau\rangle dt
\]
and
\[
\|I\|_{L^2(\mathbb{R},dk)}\lesssim T^{-\gamma} \int_0^T
\left\|\int_t^T e^{-ikT^{-2}\Lambda
\tau}\hat{V}(\tau)e^{ikT^{-2}\Lambda \tau
}\delta_0d\tau\right\|_{L^2(\mathbb{R},dk)} dt
\]
By Plancherel,
\[
\left(\int_\mathbb{R} \|I\|_2^2dk\right)^{1/2}\lesssim
T^{5/2-\alpha-2\gamma}
\]
due to the limitations on the support of $\hat{V}$.
\end{proof}
\end{lemma}
Clearly, by taking $\alpha+2\gamma>5/2$, we have localization of
almost all of the $L^2$--norm on the first harmonic for most $k$ but
this argument does not say much about the Sobolev norms.
\bigskip\bigskip\bigskip

 In the rest of this section, we will focus on
(\ref{tough}) with short range potential, e.g.
$V(t,\theta)=\cos(\theta)q(t)$. For simplicity, we start with the
following problem where all eigenvalues are non-degenerate
\begin{equation}
x_t=ik\Lambda x+iQ x, \quad x(0,k)=\delta_0 \label{model3}
\end{equation}
where $\Lambda$ is diagonal with eigenvalues $\lambda_n=n^2$,
$n=0,1,\ldots$ and $Q$ is symmetric Toeplitz operator:
$Q_{mn}(t)=q_{m-n}(t), q_0(t)=0, q_{-m}(t)=\bar{q}_m(t)$, $m,n\geq
0$.

We will use the following notations: given a function $v(t)$, let
$\sigma_\alpha(t)=\langle t\rangle ^{-1-\alpha}+\langle
t\rangle^{-\alpha}|v(t)|$ where $\langle t\rangle=(1+t^2)^{1/2}$ and
$\alpha\geq 0$ is to be specified later.
\begin{theorem}
Assume that ${q}_n(t)=v(t)(\delta_{-1}+\delta_{1})$ and
$|v(t)|\lesssim t^{-\gamma}, \gamma>3/4$. Then, for a.e. $k$, we
have
\[
\sup_{t>0} \sum_{n\geq 0} n^{s} |x(t,n,k)|^2<\infty, \quad \forall
s\in \mathbb{N}
\]
Fix any $(a,b)$ not containing $0$. Then, for any $s\geq 1$
\begin{equation}
\left\|\sup_{t\leq T} \sum_{n=1}^\infty
n^s|x_n(t,k)|^2\right\|_{L^{2/s}(a,b)}\lesssim C_1^s(T)C_2^s(T)+1
\label{wek}
\end{equation}
Here,
\[
C_1(T)=\left(\int_0^T \langle
\tau\rangle^{2\alpha}v^2(\tau)d\tau\right)^{1/2},\quad C_2(T)=
\int_0^T \sigma_\alpha(\tau) d\tau, \quad 1-\gamma<\alpha<\gamma-1/2
\]\label{short}
\end{theorem}
\begin{proof}
We have
\[
x_n'(t,k)=iv(t)x_{n-1}(t,k)+ikn^2 x_n(t,k)+iv(t)x_{n+1}(t,k),\quad
n>0
\]
and
\[
x_0'(t,k)=iv(t)x_1(t,k), \quad \quad x_n(0,k)=\delta_0
\]
Thus, we have
\[
S_N(T)=\sum\limits_{n=N}^\infty |x_n(T,k)|^2=-2\Im
\left[\int\limits_0^T v(t)x_{N-1}(t,k)\bar{x}_N(t,k)dt\right], N>1
\]
Writing $x_n=\exp(ikn^2t)\psi_n$, we have
\[
\psi'_n=iv(t)\exp(-ikn^2t)(x_{n-1}+x_{n+1})
\]
and so
\[
S_N(T)\lesssim \left| \int_0^T v'_N(t,k)\langle t\rangle
^{-\alpha}\psi_{N-1}\bar{\psi}_Ndt\right|
\]
where
\[
v_N(t,k)=-\int_t^T \langle \tau\rangle ^\alpha
v(\tau)\exp(-ik(2N-1)\tau)d\tau
\]
Taking $N>2$ and integrating by parts,
\begin{eqnarray}
S_N(T) \lesssim \int\limits_0^T |v_N(t,k)|\cdot (\langle t\rangle
^{-1-\alpha}+\langle t\rangle ^{-\alpha}|v(t)|)\cdot\quad\quad\quad\quad\quad\nonumber\\
(|x_{N-2}x_N|+|x_N|^2+|x_{N-1}|^2+|x_{N-1}x_{N+1}|)dt \label{core}
\end{eqnarray}
Notice that for any $t$, we have
\[
|v_N(t,k)|\lesssim M(k(2N-1))
\]
where $M(k)$ is Carleson-Hunt maximal function for $\langle
t\rangle^\alpha v(t)$ and $M(k)\in L^2(\mathbb{R})$. Let
\[
\mu(k)=\left(\sum_{n=1}^\infty |M(kn)|^2\right)^{1/2}
\]

 For $0<a<b$, we have
\begin{eqnarray*} \int_a^b \mu^2(k)dk\lesssim \sum_{m=1}^\infty
\sum_{n=2^m}^{2^{m+1}} n^{-1} \int_{\alpha n}^{\beta
n}|M(\xi)|^2d\xi  \lesssim \sum_{m=1}^\infty \int_{\alpha
2^m}^{\beta 2^{m+1}}|M(\xi)|^2d\xi \\\lesssim \|M\|^2\lesssim
C_1^2(T)
\end{eqnarray*}
Thus, by Fubini, we have $M(kn)\in \ell^2(\mathbb{Z}^+)\subset
\ell^\infty (\mathbb{Z}^+)$ for a.e. $k$.
\begin{equation} S_N(T)
\leq \mu(k)\int\limits_0^T (\langle t\rangle ^{-1-\alpha}+\langle
t\rangle^{-\alpha}|v(t)|)
(|x_{N-2}x_N|+|x_N|^2+|x_{N-1}|^2+|x_{N-1}x_{N+1}|)dt \label{ite}
\end{equation}
Sum these inequalities over $N$ using  $\|x(t,k)\|_2=1$ for any $t$
\[
\sup_{t\leq T} \sum_{n=1}^\infty n|x_n(t,k)|^2\lesssim C_2(T)
\mu(k)+1
\]
By induction,
\begin{equation}
\sup_{t\leq T} \sum_{n=1}^\infty n^s|x_n(t,k)|^2\lesssim
C_2^s(T)\mu^s(k)+1 \label{sobn}
\end{equation}
for any $s$. Thus, there is a full measure set such that
\[
\sup_{t>0} \sum_{n=1}^\infty n^s|x_n(t,k)|^2<\infty
\]
for any $s$. Integration of (\ref{sobn}) gives (\ref{wek}).
\end{proof}
 We also can improve this result to get real analyticity for a.e. $k$.

\begin{proposition}
Under the conditions of the theorem \ref{short}, there is a full
measure set in $k$ for which the solution is real analytic.
\end{proposition}

\begin{proof}
We will work on the Fourier side.  Summing (\ref{ite}) from $N=2$ to
$\infty$
\begin{equation}
\sum_{N=2}^\infty |x_N(T,k)|^2(N-1)\leq C\mu(k)\int\limits_0^T
\sigma_\alpha(t)dt \label{it1}
\end{equation}
Multiply (\ref{ite}) by $N-3$ and sum from $N=4$ to $\infty$.
(\ref{it1}) gives
\[
\sum_{N=4}^\infty |x_N(T,k)|^2(N-3)^2\leq C^2\cdot
2\mu^2(k)\int\limits_0^T \sigma_\alpha(t_1)\int\limits_0^{t_1}
\sigma_\alpha(t_2)dt_2dt_1
\]
By induction
\[
\sum_{N=2l}^\infty |x_N(T,k)|^2(N-(2l-1))^l\leq (C\mu(k))^l
\left(\int_0^T \sigma_\alpha(t)dt\right)^l
\]
Taking, say,  $N\sim 4l$, we have
\[
\sup_{t\geq 0} |x_{N}(t,k)|^2\leq
\left(\frac{C\mu(k)\|\sigma_\alpha\|_1}{l}\right)^l
\]
which shows that the solution is real analytic for a.e. $k$.
\end{proof}

In theorem \ref{short}, the integration is restricted to an interval
$(a,b)$ which must be finite, not containing $0$. Below we show that
this condition can be dropped.

\begin{theorem} Under the conditions of theorem \ref{short}, we have
\begin{equation}
\sup_{t>0} \sum_{n=1}^\infty
n^2|x_n(t,k)|^2 \in L^1_{\rm loc}(\mathbb{R}) \label{pust}
\end{equation}

\end{theorem}
\begin{proof}
Notice that the function $\langle t\rangle^\alpha v(t)\in
L^\nu(\mathbb{R}^+)$ for some $\nu(\gamma)<2$ and therefore $M(k)\in
L^\zeta(\mathbb{R})$ with $\zeta$ dual to $\nu$. Multiply
(\ref{core}) by $N$ and sum from $N=2$ to infinity. We have
\[
I(T,k)=\sum_{n=1}^\infty
n^2|x_n(T,k)|^2\lesssim 1+ \quad\quad\quad
\] \[ \int\limits_0^T \sigma_\alpha(t)
 \int\limits_\mathbb{R} \sum_{n\geq 2}
n^{-\epsilon}|M((2n-1)k)|\cdot
n^{1+\epsilon}\left( |x_{n-2}(t,k)x_n(t,k)|+|x_n(t,k)|^2\right.
\]
\[
\left. +|x_{n-1}(t,k)|^2+|x_{n-1}(t,k)x_{n+1}(t,k)|\right)dt
\]
where $\epsilon>0$. By Young's inequality, we have
\[
I(T,k)\lesssim 1+ \int\limits_0^T \sigma_\alpha(t)
 \sum_{n\geq 2} \left(
\frac{n^{-\zeta\epsilon}|M((2n-1)k)|^\zeta}{\zeta} + \frac{
n^{\nu(1+\epsilon)}|x_{n-2}(t,k)|^{2\nu}}{\nu}\right) dt
\]
 Taking $\epsilon=(2-\nu)/\nu $, we get
\[
I(T,k)\lesssim  1+A(k)+\int_0^T \sigma_\alpha(t)I(t,k)dt
\]
where
\[
A(k)= \left(\int\limits_0^T \sigma_\alpha(t)dt\right) \cdot
 \left( \sum_{n\geq 2}
n^{-\zeta\epsilon}|M((2n-1)k)|^\zeta \right)\in L^1(\mathbb{R})
\]
The Gronwall lemma yields
\[
I(T,k)\lesssim (1+A(k))\exp\left( C_2(T)\right)
\]
which implies (\ref{pust}).
\end{proof}
The similar argument can handle the higher Sobolev
norms.\bigskip\bigskip

The next theorem  studies the $L^p(\mathbb{R}, dk)$ norms of
\[
S_N(T,k)=\sum_{n=N}^\infty |x_n(T,k)|^2
\]
\begin{theorem}
 Assume that conditions of the theorem \ref{short} hold. Then, for any \mbox{$
2\leq p\leq \infty,$}  $N>1$, we have
\begin{equation}
\|S_N(T,k)\|_p \lesssim N^{-2+2p^{-1}}\left(\int\limits_0^T |v(t)|
dt\right)^{2-2/p} \left(\int\limits_0^T v^2(\tau)d\tau\right)^{1/p}
\end{equation}
\end{theorem}
\begin{proof}
We have
\begin{eqnarray*}
S_m(T,k)\lesssim \int\limits_0^T |v(t)|\left|\int\limits_t^T
v(\tau)e^{i(2m-1)k\tau}d\tau\right|(|x_{m-2}x_m|+|x_m|^2+
\hspace{2cm}
\\
+|x_{m-1}|^2+|x_{m-1}x_{m+1}|)dt
\end{eqnarray*}
Sum these inequalities in $m$ from $N/2$ to $N$. We get
\[
NS_N(T,k)\lesssim \int\limits_0^T |v(t)| \max\limits_{m=N/2, \ldots,
N}\left| \int\limits_t^T v(\tau)e^{i(2m-1)k\tau}d\tau\right|dt
\]
Taking the $L^2(\mathbb{R},dk)$ norm of both sides, we have by
Minkowski
\[
\|S_N\|_2\lesssim N^{-1}\int\limits_0^T |v(t)|
\left(\sum\limits_{m={N/2}}^N  \int\limits_\mathbb{R}
\left|\int\limits_t^T v(\tau)e^{i(2m-1)k\tau} d\tau \right|^2dk
\right)^{1/2}dt
\]
so
\begin{equation}
\|S_N\|_2\lesssim N^{-1}\int\limits_0^T |v(t)| \left(\int\limits_t^T
|v(\tau)|^{2}d\tau\right)^{1/{2}}dt, \label{inte1}
\end{equation}

The argument similar to the one employed in the proof of lemma
\ref{prim1} gives
\[
\sum_{n\geq 0} n^2|x_n|^2\lesssim \left(\int_0^T |v(t)dt\right)^2
\]
uniformly in $k$. Thus,
\begin{equation}
\|S_N\|_\infty \lesssim N^{-2}\left(\int_0^T|v(t)|dt\right)^2
\label{inte2}
\end{equation}
Interpolation between (\ref{inte1}) and (\ref{inte2}) gives the
statement of the theorem.
\end{proof}

Repeating the same arguments for the case when the eigenvalues $\{\lambda_j\}, j>0$ have multiplicity two, one has
\begin{theorem}
Let $\psi(t,\theta,k)$ be the solution to (\ref{tough}) and
$q(t,\theta)=\cos(\theta)q(t)$  where $|q(t)|\lesssim t^{-\gamma},
\,\gamma>3/4$. Then
\begin{itemize}
\item[1.]
For a.e. $k$ we have
\[
\sup_{t>0} \|\psi(t,\theta,k)\|_{H^s(\mathbb{T})}<\infty, \quad s\in \mathbb{Z}^+
\]
 and $\psi(t,\theta,k)$ is real analytic in $\theta$ for any $t$.
\item[2.] For any finite interval $(a,b)$ not containing zero,

\[
\sup_{t>0} \|\psi(t,\theta,k)\|^2_{H^{s/2}(\mathbb{T})}\in
L^{2/s}(a,b)
\]

\item[3.] If $S_N(T,k)=\|P_{|n|\geq N}\psi(T,\theta,k)\|^2$, then
\[
\|S_N(T,k)\|_p \lesssim N^{-2+2p^{-1}}\left(\int\limits_0^T |q(t)|
dt\right)^{2-2/p} \left(\int\limits_0^T q^2(\tau)d\tau\right)^{1/p},
\quad 2\leq p\leq \infty
\]
\end{itemize}
\end{theorem}

\bigskip\bigskip

Consider the model (\ref{dt1}) with potential
$V(t,\theta)=q(t)\cos(\mu \theta)$ where $\mu$ is integer and
$\mu\sim T^\beta$, $\beta\in [0,1]$. Then, obviously,
$u(T,\theta,k)=\phi(T,\mu\theta,k\mu^2T^{-2})$ and
\[
i\phi_t=ik\phi_{\theta\theta}+iq(t)\cos(\theta)\phi,\quad
\phi(0,\theta)=1
\]
We have

\[
\int\limits_{\mathbb{R}} \left(\sum\limits_{|n|>N} |\phi_n(T,k)|^2
\right)^2dk\lesssim N^{-2}T^{3-4\gamma}
\]
Taking $N\sim T\mu^{-1}$, we have
\[
\int\limits_{\mathbb{R}} \left(\sum\limits_{|n|>T} |u_n(T,k)|^2
\right)^2dk\to 0
\]
for the original solution (as long as $\gamma>3/4$).\bigskip

The methods developed in this section can handle the case  of
transport equation or equation with the symbol $|n|$. Some of them
are applicable to the general short-range potentials $V$ as well.
The perturbation arguments at some places are taken from \cite{kis}.
\bigskip

In conclusion, we will mention the case for which rather
satisfactory results can be obtained. Consider the following short
range evolution
\begin{equation}
u_t=kT^{-\alpha}u_\theta+2iq(t)\cos(\theta)u, \quad
u(0)=f(\theta),\quad 0<t<T\label{ur1}
\end{equation}
where $0<\alpha<1$. Notice that $q$ is not necessarily real-valued
but we require $|q|\lesssim T^{-\gamma}$. We will be interested in
the case $\gamma<(1+\alpha)/2$.

The scaled solution is
\[
u(t,\theta-kT^{-\alpha}t,k)=f(\theta)\exp\left(
izQ(k,t)+i\bar{z}Q(-k,t) \right)
\]
where
\[
z=e^{i\theta}, \quad Q(k,t)=\int_0^t q(t)\exp(ikT^{-\alpha}t)dt
\]
We have
\[
\left\|\max_{t\in [0,T]} |Q(k,t)|\right\|_2\lesssim
T^{(1+\alpha)/2-\gamma}
\]
so for most $k$,
\[
\max_{t\in [0,T]} |Q(k,t)|\lesssim  T^{(1+\alpha)/2-\gamma}
\]

Take $k$ such that
\[
\max_{t\in [0,T]} |Q(\pm k,t)|\leq Q=CT^{(1+\alpha)/2-\gamma}
\]
and expand into Taylor series to get
\[
\exp\left( izQ(k,t)+i\bar{z}Q(-k,t) \right)= \sum_{l\in \mathbb{Z}}
z^l \alpha_l
\]
and
\[
 \quad |\alpha_l|\lesssim \sum_{j=0}^\infty
\frac{Q^{l+j}Q^j}{(l+j)!j!}, \quad l>0
\]
Notice that
\[
\frac{{Q}^{l+j}}{(l+j)!}
\]
decays in $j$ as long as $l>Q$.  So,
\[
|\alpha_l|\lesssim \frac{Q^l}{l!} e^{Q}<(d/e)^{-l} e^{Q},
\]
where $l>dQ$.  Then,
\[
\sum_{l>dQ} |\alpha_l|^2<2^{-Q}
\]
for $d$ large enough. This means exponential localization to the
range $|l|<dQ$ (since $Q$ is large) for $f=1$ and for any other column
of the monodromy matrix. Such a strong localization result allows us
to run a simple perturbation argument. Take initial value
$f(\theta)=e^{ij\theta}$ with large positive $j$. Fix $k$  such that
the localization property holds and act on (\ref{ur1}) with Riesz
projection
\[
(Pu)_t=kT^{-\alpha}(Pu)_\theta+2iPq(t)\cos(\theta)Pu+\psi
\]
where
\[
\psi=2iPq(t)\cos(\theta)P^\perp u
\]
Due to strong localization, we have
\[
\|\psi\|\lesssim T^{-\gamma}2^{-Q/2}
\]
for each $t\in [0,T]$ provided that $j>dQ$ so $Pu$ is an approximate
solution to the problem
\begin{equation}
y_t=ikT^{-\alpha}|\partial|y+2iq(t)P\cos\theta Py, \quad y(0)=e_j
\label{ur4}
\end{equation}
Assuming that $q$ is real-valued and using Duhamel formula, we get
\[
\max_{t\in [0,T]} \|Pu-y\|\lesssim T^{1-\gamma}2^{-Q/2}
\]
In particular, that means $\max_{t\in
[0,T]}|x_{0l}(t,k)|<T^{1-\gamma}2^{-Q/2}$ for each $l:dQ<l<T$ where
$x_{ij}$ are elements of the monodromy matrix for the problem
(\ref{ur4}). Due to symmetry of the monodromy matrix (or time
reversal), we have
\[
\max_{t\in [0,T]}\sum_{l=dQ}^T |x_{l0}(t,k)|^2\lesssim
T^{3-2\gamma}2^{-CT^{\epsilon}}, \quad
\epsilon=(1+\alpha)/2-\gamma>0
\]
Since the first column is always localized to the range
$(0,T^{1-\gamma})$, we obtain its localization to the range
$(0,T^{1-\gamma}(dT^{(\alpha-1)/2}))$ for most $k$. By simple
scaling, one proves that the solution to
\begin{equation}
y_t=ikT^{-1}|\partial|y+iq(t)P\cos(\mu\theta) Py, \quad y(0)=1,
\quad \mu\leq T
\end{equation}
is localized to $[0,T^{1-\gamma}\sqrt{\mu}]$ for most $k$ versus
$[0,T^{1-\gamma}{\mu}]$ for all $k$. If $\gamma\geq 1/2$, then
$[0,T^{1-\gamma}\sqrt{\mu}]\subseteq [0,T]$, as expected.

\section{Appendix}
In this section, we collect rather standard results that we used in
the main text. The following lemma is well-known
\begin{lemma}
If $f\in H^{1/2}(\mathbb{T})$, then
\[
\|e^{f}\|_2<C_1 e^{C_2\|f\|_{H^{1/2}(\mathbb{T})}^2}
\]
Also, this map is continuous on $H^{1/2}(\mathbb{T})$.
 \label{al1}
\end{lemma}
\begin{proof}
We have
\[
e^f=\sum_{n=0}^\infty \frac{f^n}{n!}
\]
\[
\|f^n\|_2=\|\hat{f}\ast \ldots \ast \hat{f}\|_2
\]
By H\"older,
\[
\|\hat{f}\|_p\lesssim
\left(\frac{C}{p-1}\right)^{(2-p)/(2p)}\|f\|_{H^{1/2}(\mathbb{T})}
\]
and by Young's inequality
\[
\|\hat{f}\ast \ldots \ast \hat{f}\|_2\leq \|\hat{f}\|^n_{p_n},\quad
p_n=2n(2n-1)^{-1}
\]
So,
\[
\|f^n\|_2\lesssim (Cn)^{n/2}\|f\|_{H^{1/2}(\mathbb{T})}^n
\]
which yields the necessary bound after application of Stirling's
formula. Also, since each term in the series is continuous in $f$,
we have continuity of the exponential map.
\end{proof}

\begin{lemma}
If $f\in H^{1/2}(\mathbb{T})$ and $f\in \mathbb{R}$, then
\[
\|e^{if}-1\|_{H^{1/2}(\mathbb{T})}\lesssim
\|f\|_{H^{1/2}(\mathbb{T})}
\]
and the exponential map is continuous in $H^{1/2}(\mathbb{T})$
metric.\label{al2}
\end{lemma}
\begin{proof}
We have the following characterization of $H^{1/2}$ space
(\cite{Sim1}, Propositions 6.1.10 and 6.1.11)
\begin{equation}
\|f\|_{H^{1/2}(\mathbb{T})}^2\sim
|\hat{f}(0)|^2+\int\limits_{\mathbb{T}}\int\limits_{\mathbb{T}}
\frac{|f(x)-f(y)|^2}{|x-y|^2}dxdy \label{ch}
\end{equation}
Since
\[
\left|\int\limits_\mathbb{T}
(e^{if(x)}-1)dx\right|^2\leq\left|\int\limits_\mathbb{T}
|f(x)|dx\right|^2\lesssim \|f\|_2^2\leq
\|f\|_{H^{1/2}(\mathbb{T})}^2
\]
and
\[
|e^{if(x)}-e^{if(y)}|=\left|\int\limits_{f(x)}^{f(y)}
e^{it}dt\right|\leq |f(x)-f(y)|
\]
we have the first statement of the lemma. The continuity of
exponential at zero is elementary. Now, assume that
$\|f_n-f\|_{H^{1/2}}\to 0$. Clearly,
\[
\int_\mathbb{T} e^{if_n}dx\to \int_\mathbb{T}e^{if}dx
\]

 For the second term in (\ref{ch}), we have
\[
{e^{if_n(x)}-e^{if(x)}-(e^{if_n(y)}-e^{if(y)}})=(e^{i(f_n(x)-f(x))}-e^{i(f_n(y)-f(y))})e^{if(x)}
\]
\[
+(e^{i(f_n(y)-f(y))}-1)(e^{if(x)}-e^{if(y)})
\]
and we just need to show that
\[
\int\limits_{\mathbb{T}}\int\limits_{\mathbb{T}}
\frac{|(e^{i(f_n(y)-f(y))}-1)(e^{if(x)}-e^{if(y)})|^2}{|x-y|^2}dxdy\to
0
\]
The function $ F_n(y) =e^{i(f_n(y)-f(y))}-1$  satisfies $|F_n|\leq
2$ and $\|F_n\|_1\to 0$. So $F_n=F^1_n+F^2_n$ such that
$F^1_n=F_n\cdot \chi_{|F_n|<\epsilon}, |F_n^1|<\epsilon$ and

\[
|F^2_n|\leq 2, \quad |{\rm supp}\,(F^2_n)|\lesssim
\epsilon^{-1}\int\limits_{\mathbb{T}} |F_n|dx\to 0
\]
Since $\epsilon$ is arbitrary positive number,
\[
\int\limits_{\mathbb{T}}\int\limits_{\mathbb{T}} \frac{|F_n(y)(e^{if(x)}-e^{if(y)})|^2}{|x-y|^2}dxdy\to
0,\quad n\to \infty
\]
\end{proof}

 {\bf Acknowledgements.} This research was supported by Alfred
P. Sloan Research Fellowship and NSF Grants: DMS-0500177,
DMS-0758239.

\bigskip


\begin{thebibliography}{99}

\bibitem{as} M. Abramowitz, I. Stegun,  Handbook of mathematical functions with formulas,
graphs, and mathematical tables. National Bureau of Standards
Applied Mathematics Series, 55


\bibitem{BS} A. B\"ottcher, B. Silbermann, Introduction to large
truncated Toeplitz matrices, Springer, 1998

\bibitem{b1} J. Bourgain,  On growth of Sobolev norms in linear
Schr\"{o}dinger equations with smooth time dependent potential. J.
Anal. Math. 77 (1999), 315--348







\bibitem{Den1} S. Denisov, Continuous analogs of polynomials
orthogonal on the unit circle. Krein systems, Int. Math. Res.
Surveys, Vol. 2006 (2006)

\bibitem{den2} S. Denisov, An evolution equation as the WKB correction in long-time asymptotics
of Schr\"odinger dynamics, Comm. Partial Differential Equations,
Vol. 33, N2, 2008, 307-319


\bibitem{Carl} O. J{\o}rsboe, L. Mejlbro, The Carleson-Hunt Theorem on
Fourier Series, Lecture Notes in Mathematics, 911, Springer, 1982

\bibitem{kis} A. Kiselev, Stability of the absolutely
continuous spectrum of the Schr\"{o}dinger equation under slowly
decaying perturbations
 and a.e. convergence of integral operators. Duke Math. J. 94 (1998), no. 3, 619--646

\bibitem{Nen} G. Nenciu, Adiabatic theory: stability of
systems with increasing gaps. Ann. Inst. H. Poincare Phys. Theor. 67
(1997), no. 4, 411--424


\bibitem{pert} B. Perthame,  Mathematical
tools for kinetic equations. Bull. Amer. Math. Soc. (N.S.) 41 (2004), no. 2, 205--244

\bibitem{Sim1} B. Simon, Orthogonal polynomials on the unit circle.
Parts  1 and  2.  American Mathematical Society Colloquium
Publications, 54,  American Mathematical Society, Providence, RI,
2005

\bibitem{w1} W.-M. Wang,  Bounded Sobolev norms for
linear Schr\"{o}dinger
 equations under resonant perturbations. J. Funct. Anal. 254 (2008), no. 11, 2926--2946
\end{thebibliography}
\end{document}